\documentclass[a4paper]{article}
\usepackage[all]{xy}\usepackage[latin1]{inputenc}        
\usepackage[dvips]{graphics,graphicx}
\usepackage{amsfonts,amssymb,amsmath,color,mathrsfs, amstext}
\usepackage{amsbsy, amsopn, amscd, amsxtra, amsthm,authblk}
\usepackage{enumerate,algorithmicx,algorithm}
\usepackage{algpseudocode}
\usepackage{upref}
\usepackage{geometry}
\usepackage{amsthm,amsmath,amssymb}
\usepackage{relsize}
\usepackage{mathrsfs}
\usepackage{bm}
\usepackage{exscale}
\usepackage{graphicx}
\usepackage{tabularx}
\usepackage{threeparttable,multirow,dcolumn,booktabs}
\usepackage{algorithmicx,algorithm}
\usepackage{tabularx}
\usepackage{caption}
\usepackage[displaymath]{lineno}
\usepackage{float,color}
\usepackage{bm}
\usepackage{caption}
\usepackage{setspace}
\usepackage{yhmath}
\usepackage{booktabs}
\usepackage{subcaption}
\usepackage{multirow}
\usepackage{makecell}
\usepackage[colorlinks,
            linkcolor=red,
            anchorcolor=red,
            citecolor=red
            ]{hyperref}

\numberwithin{equation}{section}

\newtheorem{theorem}{Theorem}[section]
\newtheorem{lemma}{Lemma}[section]

\newtheorem{remark}{Remark}[section]

\numberwithin{equation}{section}

\makeatletter
\newcommand{\biggg}{\bBigg@{3}}
\newcommand{\Biggg}{\bBigg@{3.5}}
\newcommand{\bigggg}{\bBigg@{4}}
\newcommand{\Bigggg}{\bBigg@{4.5}}
\newcommand{\biggggg}{\bBigg@{5}}
\newcommand{\Biggggg}{\bBigg@{5.5}}
\makeatother

\begin{document}

\title{A Kernel-Independent Sum-of-Exponentials Method}

\author[1]{Zixuan Gao\thanks{1270157606gzx@sjtu.edu.cn}}
\author[1,2]{Jiuyang Liang\thanks{liangjiuyang@sjtu.edu.cn}}
\author[1,2]{Zhenli Xu\thanks{xuzl@sjtu.edu.cn}}

\affil[1]{School of Mathematical Sciences, Shanghai Jiao Tong University, Shanghai 200240, P. R. China}
\affil[2]{Institute of Natural Sciences and MOE-LSC, Shanghai Jiao Tong University, Shanghai, 200240, P. R. China}

\date{}
\maketitle

\begin{abstract}
We propose an accurate algorithm for a novel sum-of-exponentials (SOE) approximation of kernel functions,
and develop a fast algorithm for convolution quadrature based on the SOE,  which allows an order $N$ calculation for $N$ time steps of approximating a continuous temporal convolution integral.
The SOE method is constructed by a combination of the de la Vall\'ee-Poussin sums for a semi-analytical exponential expansion of a general kernel, and a model reduction technique for the minimization of the number of exponentials under given error tolerance. We employ the SOE expansion for the finite part of the splitting convolution kernel such that the convolution integral can be solved as a system of ordinary differential equations due to the exponential kernels. The significant features of our algorithm are that the SOE method is efficient and accurate, and works for general kernels with controllable upperbound of positive exponents.  We provide numerical analysis for both the new SOE method and the SOE-based convolution quadrature. Numerical results on different kernels, the convolution integral and integral equations demonstrate attractive performance of both accuracy and efficiency of the proposed method.


\end{abstract}

\section{Introduction}
In this paper, we describe a new numerical algorithm for the construction of sum-of-exponentials (SOE) approximation for a given scalar smooth function $f(x)$ of the form
\begin{equation}\label{SOEexpansion}
	\max\limits_{x\in I}\left|f(x)-\sum_j m_j e^{-s_j x}\right|<\varepsilon,
\end{equation}
where $I$ is a finite interval which could be an arbitrary subset of $\mathbb{R}^+$, $\varepsilon$ is an error tolerance, and $m_j$ and $s_j$ are parameters representing the weight and exponent of the $j$th exponential, respectively.
In Eq.\eqref{SOEexpansion}, both $m_j$ and $s_j$ could be complex, while the real part of $s_j$ is assumed to be non-negative to avoid the explosion at infinity.

Over the past decades, SOE methods has attracted attention in many applications of scientific computing \cite{Beylkin2009Fast,H1999A,Yong2018The,Jiang2008Efficient,Lubich2002},
as an SOE approximation enables a recurrence scheme to compute the spatial discrete convolution, and thus dramatically reduces the cost. In some engineering areas, like acoustic and electromagnetic simulations, SOE methods are also involved as an useful tool in constructing discrete complex image representation, which plays an important role in developing fast multipole method (FMM) for 3D Helmholtz Green's function at low frequencies in layered media \cite{wang2020taylor}, by applying the Sommerfeld identity \cite{fang1988discrete,alparslan2010closed}. Employing the SOE for constructing fast algorithm for the calculation of pair interaction is also explored in some recent works \cite{spivak2010fast,2020onekernel}.

How well the SOE with applications to these areas depends on three key factors: the convergence rate of the SOE with the increasing number of exponentials, the robustness especially for kernels with near-singular and moderate oscillating nature, and the range of the maximum (in the sense of module) exponent of the SOE, i.e., $\max_{j}|s_j|$ in Eq.\eqref{SOEexpansion}. The design of such an SOE approximation for some additional constraints is strongly nonlinear and highly nontrivial, and has been an extensively studied subject \cite{Beylkin2005Approximation,Beylkin2010Approximation,Braess1995Asymptotics,Dietrich2005Approximation,Braess2009On,Evans1980On,articleGonchar,jiang2019fast,hamming2012numerical,Wiscombe1977Exponential}. One less studied but important problem is the last issue, i.e., the approximation with small exponents. In fact, a large maximum exponent may take serious effect on both the accuracy and efficiency, as it has a large derivative and the roundoff error where the former will lead to a stiff problem \cite{boyd2010uselessness,Lubich1993Runge} and the latter will increase with the rise of exponent. This problem will be further explained in the Section \ref{fastconv} and Section \ref{convolutionquadrature} of this paper, taking examples of convolution quadrature and solving integral equations. Besides, theoretical discussions on the convergence rate are also difficult to be established in some SOE methods based on adaptive partition.


Motivated by above observations, we propose a novel kernel-independent and high-accurate SOE method by a combination of the de la Vall\'ee-Poussin (VP) sum \cite{Al2001Approximation,de1919leccons} and the model reduction (MR) \cite{Moore1981Principal,Peter2015SIAM} technique. In the so-called VPMR approach, an accurate SOE approximation is first constructed by employing the VP sum via a variable substitution,
\begin{equation}\label{variable1}
	x=-n_c\log\left(\dfrac{1+\cos r}{2}\right),~r\in[0,\pi],
\end{equation}
where the parameters of the exponentials are expressed analytically. The variable substitution introduces a parameter $n_{c}$ which allows to tune the maximal exponent of the exponentials. Subsequently, the MR technique is used to further reduce the number of exponentials within the given error tolerance based on an observation that
\begin{equation}\label{laplace}
	\mathscr{L}\left[\sum_{j}m_{j}e^{-s_jx}\right]=\sum_{j}\dfrac{m_{j}}{z+s_j}
	=\bm{c}(z\bm{I}-\bm{A})^{-1}\bm{b},
\end{equation}
where $\mathscr{L}$ denotes the Laplace transform, $\bm{A}$ is a diagonal matrix, $\bm{b}$ and $\bm{c}$ are column and row vectors, respectively. The right hand of Eq.\eqref{laplace} shares the same form of the transfer function of a linear dynamical system, thus the so-called balanced truncation method \cite{Moore1981Principal} can be used for the MR, achieving an optimized SOE approximation. The universal error estimation and complexity analysis are both furnished. The error estimation validates that Eq.\eqref{variable1} will preserve the smoothness of the resulting function at the origin, and guarantees the global convergence rate. The comparison of accuracy and convergence rate between VPMR and several theoretical methods, including contour integral methods \cite{jiang2019fast} on Gaussian kernel and the classical Prony's method \cite{hamming2012numerical} on four frequently-used kernels which have complicated forms, near-singular properties, or oscillations, show promising efficiency of our method.

The SOEs of the form Eq.\eqref{SOEexpansion} occur commonly in computational mathematics and computational physics. For instance, \cite{schadle2006fast} describes a scheme for accelerating the convolution quadrature,
\begin{equation}\label{conSOE}
	y(t)=f*g=\int_0^tf(t-\tau)g(\tau)d\tau,
\end{equation}
where $f(\tau)$ is the kernel function, which may has singularity at the origin, and $g(\tau)$ is a smooth function, based on a technique of the Laplace inverse transform of the kernel function such that
\begin{equation}
	y(t)=\dfrac{1}{2\pi i}\int_\Gamma F(\lambda)\int_0^t e^{\lambda(t-\tau)}g(\tau)d\tau d\lambda.
\end{equation}
The scheme requires the construction of quadratures of $\lambda$ for which the inverse Laplace transform is well approximated, and the resulting convolution integral associated with $\tau$ satisfies a combination of ODEs with form $u'=\lambda u+g(t)$ with $u(0)=0$ and can be integrated by the RK method. A stable and efficient discretization of the interval $\Gamma$ requires an $O(1/h)$ cutoff of both the upper/lower bound of $\Gamma$ with $h$ the step size of the RK method. This $O(1/h)$ cutoff will lead to a stiff problem such that the convergence rate is $O(h^{\min(p,q+1)})$ with $p$ the order and $q$ the stage order. We describe in detail the construction of efficient quadratures for the discretization of Eq.\eqref{conSOE} in Section \ref{fastconv} of this paper, by first replacing the inverse Laplace transform of $f(t-\tau)$ to the corresponding SOE produced via the VPMR method. Second, for the case of $f(t-\tau)$ having no singularity, we follow the general procedure of the RK method to evaluate $y(t)$ recursively; for the case of $f(t-\tau)$ having singularity at the origin, the singularity part of the convolution integral is extracted and is approximated by combination of the local expansion of $f(t-\tau)$ and the local interpolation of $g(\tau)$, and then the RK to the smooth part can be applied. Rigorous error bound is established by following the work of Lubich \emph{et al.} \cite{Lubich1993Runge,banjai2011error}, showing $O(h^{p})$ rate of the RK which the stage order is nolonger involved.

SOEs also have obvious applications to the discretization of integral equations. In particular, for a given kernel function $f$ convoluting with an unknown (except the initial point) function, given integral equations of the forms either
\begin{equation}\label{Ht1}
	(1-\varpi)g(t)+H(t)=\int_0^tf(t-\tau)g(\tau)d\tau
\end{equation}
with $H(\tau)$ a given source, $\varpi$ a real parameter,
or the so-called nonlinear Volterra equation
\begin{equation}\label{Volterra1}
	u(t)=a(t)+\int_0^tf(t-\tau)g(\tau,u(\tau))d\tau
\end{equation}
with $a(\tau)$ an inhomogeneous known function, the discretization \cite{schadle2006fast} based on the inverse Laplace transform of $f(t-\tau)$ proceeds by choosing appropriate integral contour such that the convolution part of Eq.\eqref{Ht1} and Eq.\eqref{Volterra1} are approximated by
\begin{equation}
	\sum_{j=0}^{n}W_{n-j}g_{j},~\,\,
\end{equation}
with $t=nh$, $W$ the weight matrices, and $g_j$ the numerical value of $g$ at $\tau=jh$. Also, the RK method is used for recursively solving the resulting linear/nonlinear equation, and a large $O(1/h)$ cutoff of the contour will reduce the performance especially for a small step size. In Section \ref{convolutionquadrature}, our schemes are developed to construct the discretization of these integral equations by using the SOE of $f(t-\tau)$ produced by the VPMR method. These schemes provide an efficient means of avoiding stiff problem when solving integral equations, where the convergence rate of the RK is $O(h^{p})$ and the stage order is not involved, as compared with many inverse Laplace-based works \cite{Lubich1993Runge,schadle2006fast,banjai2011error}. The error estimation of our schemes are non-trivial thanks to the introducted technique of removing singularity, and are rigorously provided.

The remainder of the paper is organized as follows. In Section 2, we introduce the VPMR method to find an efficient and accurate SOE approximation, and establish both an universal error estimation and a detailed complexity analysis. Comparison results between the VPMR and some existing SOE works on different kernels are also provided. In Section 3, we describe numerical algorithms with numerical analysis of errors, which are based on SOE and RK method, for temporal convolutions. In Section 4, we extend our SOE to construct an algorithm for fast solving two kinds of integral equations including the linear convolution equation and the nonlinear Volterra integral equation. Conclusions are given in Section 6.


\section{Sum-of-exponentials Method}\label{VPMR}
In this section, we introduce a kernel-independent SOE method based on the VP sums and the MR for constructing a VPMR algorithm for the SOE approximation.
This VPMR algorithm was originally proposed to design a sum-of-Gaussians approximation
which the minimum bandwidth is controllable for all of the non-oscillatory kernel
thus speeds up the evaluation of resulting kernel summation problem \cite{2020kernel}.
This paper extends the VPMR idea to obtain the SOE approximation by first using a different variable change which has better continuity and then applying the MR to the resulting SOE. A new rigorous error estimate for more general cases is also provided in this section.

\subsection{The VPMR algorithm}
Throughout this work, the kernel function $f(x)$ is defined on the positive
axis $x\geq0$, and has a finite limit at $x\rightarrow \infty$. We introduce a variable change
\begin{equation}\label{variable}
x=-n_c\log\left(\dfrac{1+\cos r}{2}\right),~r\in[0,\pi],
\end{equation}
such that $K(r) = f(x)$ is smooth on $[0,\pi]$, and $n_c$ is a positive number which is used to control
the upperbound of the positive exponents. The variable change is a one-to-one map and
an even and periodic prolongation of $K(r)$ can be employed such that the kernel function $K(r)$ is defined
 on the whole axis $(-\infty,\infty)$ with $2\pi$ period. The kernel can then be represented by the
 VP-sum \cite{2020kernel,Al2001Approximation} approximation, $K(r) \approx V_{n}[K(r)]$, where,
\begin{equation}\label{Vn}
V_n[K(r)]=\dfrac{2}{n\pi}\sum\limits_{\ell=n}^{2n-1}\sum\limits_{j=0}^{\ell}\alpha_j\cos(j r)\int_{0}^{\pi}K(\tau)\cos(j \tau)d\tau,
\end{equation}
with $\alpha_j=1$ for $j\geq 1$ and $\alpha_0=1/2$.
It is noted that generally the VP sums can be expressed as the means of the partial sums $S_i$, $i=0,\cdots, 2n-1$, of the Fourier series of the kernel function, $V_k[K(r)]=\sum_{i=2n-1-k}^{2n-1}S_i/(k+1)$, for which $k=0$ corresponds to the partial Fourier sums, and $k=2n-1$ corresponds to the Fej\'er sums. Here Eq.\eqref{Vn} corresponds to the $k=n$ case.

Substituting the inverse transform $r=\arccos\left(2e^{-x/n_c}-1\right)$ to Eq.\eqref{Vn},
one obtains the following approximation of kernel function,
\begin{equation}
f(x)\approx \dfrac{1}{\pi}\int_{0}^{\pi}K(\tau)d\tau+\sum_{j=1}^{2n-1}A_j T_j(2e^{-x/n_c}-1),
\end{equation}
where $T_j(x)=\cos(j\arccos x)$ is the Chebyshev polynomial of degree $j$, and
\begin{equation}\label{coefficients}
A_j=\max\left\{\dfrac{2}{\pi},\dfrac{4n-2j}{n\pi}\right\}\int_{0}^{\pi}K(\tau)\cos(j \tau)d\tau
\end{equation}
is the coefficient.

By employing the expansion form of Chebyshev polynomials, one can obtain an SOE expansion
\begin{equation}\label{firstSOE}
f(x)\approx \sum_{j=0}^{2n-1}w_je^{-jx/n_c},
\end{equation}
where the expansion coefficient $w_j$ is given by,
\begin{equation}\label{Expansion-coefficients}
w_j=\begin{cases}	
2a_0+\sum\limits_{\ell=1}^n(-1)^\ell\dfrac{n}{2n-\ell} a_\ell+\sum\limits_{\ell=1}^{n-1}(-1)^{n+\ell}\dfrac{n-\ell}{2n-\ell}a_{n+\ell},~~~~\text{for}~~j=0\\\\
2^{2j} \sum\limits_{\ell=j}^{n}(-1)^{\ell-j}\dfrac{n\ell }{(\ell+j)(2n-\ell)}\mathlarger{\binom{\ell+j}{\ell-j}}a_\ell
+\sum\limits_{\ell=1}^{n-1}c_n^{j\ell}\dfrac{n}{2n-\ell} a_{n+\ell},
~~~~\text{for}~~1\leq j\leq n\\\\	
\sum\limits_{\ell=j-n}^{n-1} \dfrac{nc_n^{j\ell}}{2n-\ell}a_{n+\ell},~~~~\text{for}~~j>n
\end{cases}
\end{equation}
and
\begin{equation}
c_n^{j\ell}=(-1)^{n+\ell-j}\left(1-\dfrac{\ell}{n}\right)\dfrac{(n+\ell)}{n+\ell+j}{\binom{n+\ell+j}{n+\ell-j}}2^{2j}.
\end{equation}

Note that the maximal positive exponent in the exponentials is $(2n-1)/n_c$, thus $n_c$ determines the  upperbound of the positive exponents.
It asymptotically becomes constant if one sets $n_{c}\propto n$. The controllability of the
 upperbound is very important because an SOE with large exponent leads to a stiff system
 when we use them to construct a fast algorithm for the temporal convolution.

One can employ the MR to reduce the number of exponentials to achieve a nearly optimal SOE approximation.
We introduce the balanced truncation method in the model order reductions \cite{Antoulas2001,Moore1981Principal} for the purpose.
Namely, we find an appropriate $P$-term exponentials to approximate Eq.\eqref{firstSOE}
\begin{equation}\label{mdrec}
\sum_{j=0}^{2n-1}w_{j}e^{-jx/n_{c}}\approx\sum_{\ell=1}^{P}m_{\ell}e^{-s_{\ell}x},
\end{equation}
with $P<2n-1$, such that the error is with a given tolerance. Based on extensive numerical experiments, we have found that the maximum positive exponent remains $\max\limits_{\ell}|s_\ell|\approx (2n-1)/n_c$ after the MR. The first step of the MR procedure is to apply the Laplace
transform on the SOE expansion (excluding $j=0$), which results in a sum-of-poles representation and
one can simply express it as a transfer function of a linear dynamical system,
\begin{equation}
\mathscr{L}\left[\sum_{j=1}^{2n-1}w_{j}e^{-jx/n_{c}}\right]=\sum_{j=1}^{2n-1}\dfrac{w_{j}}{z+j/n_c}
=\bm{c}(z\bm{I}-\bm{A})^{-1}\bm{b},
\end{equation}
where $\bm{A}$ is a $(2n-1)\times(2n-1)$ diagonal matrix, $\bm{b}$ and $\bm{c}$
are column and row vectors of dimension $(2n-1)$, respectively. The key for the MR is
to obtain the Hankel singular values by solving two Lyapunov equations,
\begin{equation}
\bm{AP}+\bm{PA}^{*}+\bm{bb}^{*}=0,~\hbox{and}~\bm{A}^{*}\bm{Q}+\bm{QA}+\bm{c}^{*}\bm{c}=0.
\end{equation}
The $i$-th Hankel singular value is defined as $\sigma_{i}=\sqrt{\lambda_{i}(\bm{PQ})}$,
where $\lambda_{i}(\bm{PQ})$ is the $i$-th eigenvalue of the product of matrices $\bm{P}$ and $\bm{Q}$.

The next step of the MR is to find a balancing transformation matrix $\bm{X}$ via the square root method \cite{Antoulas2001,Moore1981Principal}. Under this transformation and by defining matrices
$\widetilde{\bm{A}}=\bm{XAX}^{-1}$, $\widetilde{\bm{b}}=\bm{Xb}$, $\widetilde{\bm{c}}=\bm{cX}^{-1}$,
the solutions $\bm{P}$ and $\bm{Q}$ to the two Lyapunov equations,
\begin{equation}
\widetilde{\bm{A}}\bm{P}+\bm{P}\widetilde{\bm{A}}^{*}+\widetilde{\bm{b}}\widetilde{\bm{b}}^{*}=0,
~~\widetilde{\bm{A}}^{*}\bm{Q}+\bm{Q}\widetilde{\bm{A}}+\widetilde{\bm{c}}^{*}\widetilde{\bm{c}}=0,
\end{equation}
become equal and diagonal \cite{Antoulas2001},
\begin{equation}
\bm{P}=\bm{Q}=\text{diag}(\sigma_{1},\cdots,\sigma_{n}).
\end{equation}
The reduced $(P-1)$-order model is then obtained by simply taking the $(P-1)\times (P-1)$, $(P-1)\times 1$ and $1\times (P-1)$
leading blocks of $\widetilde{\bm{A}}$, $\widetilde{\bm{b}}$ and $\widetilde{\bm{c}}$, respectively.
Then the corresponding transfer function $\widetilde{\bm{c}}(zI-\widetilde{\bm{A}})^{-1}\widetilde{\bm{b}}$
satisfies \cite{Peter2015SIAM,GLOVER1984}
\begin{equation}\label{MRestimate}
\sigma_{P-1}\leq\sup\limits_{z=i\mathbb{R}}\left|\widetilde{\bm{c}}(zI-\widetilde{\bm{A}})^{-1}\widetilde{\bm{b}}
-\bm{c}(zI-\bm{A})^{-1}\bm{b}\right|\leq 2(\sigma_{P}+\sigma_{P+1}+\cdots+\sigma_{n}).
\end{equation}
Note that the transformation matrix $\bm{X}$ can be obtained via other balancing methods, such as the stochastic balancing and the positive real balancing which yields solution to appropriate Lyapunov and/or Riccati equations equal and diagonal \cite{Antoulas2001,Raimund1991}. Detailed procedures of the MR approach we use are summaried in Table \ref{modelreduction}.

\begin{algorithm}[h]
	\setstretch{1.35}
\caption{Model Reduction of Eq.\eqref{mdrec} based on balancing transformation}\label{modelreduction}
\hspace*{0.02in} {\bf Input:}$\{w_r\}_{r=0}^{2n-1}$, $n_c$ and the tolerance error $\varepsilon$.\\
\hspace*{0.02in} {\bf Output:}$\{m_l\}_{l=0}^{P-1}$ and $\{s_l\}_{l=0}^{P-1}$.
\begin{algorithmic}[1]
\State Form a diagonal matrix $\bm{A}=-\text{diag}(1/n_c,2/n_c,\cdots,(2n-1)/n_c)$, a column vector $\bm{B}=(\sqrt{|w_1|},\sqrt{|w_2|},\cdots,\sqrt{|w_{2n-1}|})^T$, and a row vector $\bm{C}=(sgn(w_1)\sqrt{|w_1|},sgn(w_2)\sqrt{|w_2|},\cdots,sgn(w_{2n-1})\sqrt{|w_{2n-1}|})$.
\State Solve the Lyapunov equations $\bm{AP}+\bm{PA}^T=-\bm{BB}^T$ and $\bm{AQ}+\bm{QA}^T=-\bm{C}^T\bm{C}$.
\State Compute the Cholesky factor $\bm{S}$ and $\bm{L}$ of the solutions of the Lyapunov equations $\bm{P}$ and $\bm{Q}$, respectively, such that $\bm{P}=\bm{SS}^T$ and $\bm{Q}=\bm{LL}^T$ hold.
\State Compute the singular value decomposition of $\bm{S}^T\bm{L}$ such that $\bm{S}^T\bm{L}=\bm{U\Sigma V}^T$, where $\bm{\Sigma}=\text{diag}(\sigma_1,\sigma_2,\cdots,\sigma_{2n-1})$ with $\sigma_i$ the so-called Hankel singular value.
\State Compute the transform matrix $\widetilde{\bm{T}}=\bm{SU\Sigma}^{-\frac{1}{2}}$.
\State Form a matrix $\widetilde{\bm{A}}=\widetilde{\bm{T}}^{-1}\bm{A}\widetilde{\bm{T}}$, a column vector $\widetilde{\bm{B}}=\widetilde{\bm{T}}^{-1}\bm{B}$, and a row vector $\widetilde{\bm{C}}=\bm{C}\widetilde{\bm{T}}$.
\State Find $P$ such that $\displaystyle 2\sum_{i=P}^{2n-1}\sigma_i\leq\varepsilon$.
\State Form a matrix $\widehat{\bm{A}}$ which denotes from the first $(P-1)\times (P-1)$ block of $\widetilde{\bm{A}}$, a column vector $\widehat{\bm{B}}$ which denotes from the frist $(P-1)$ rows of $\widetilde{\bm{B}}$, and a row vector $\widehat{\bm{C}}$ which denotes from the first $(P-1)$ columns of $\widetilde{\bm{C}}$.
\State Compute the eigenvalue decomposition $\widehat{\bm{A}}=\bm{X}\Lambda \bm{X}^{-1}$. Set $s_l=\Lambda_{ll},l=1,2,\cdots,P-1$ and $s_0=0$.
\State Compute $\mathfrak{\bm{B}}=\bm{X}^{-1}\widehat{\bm{B}}, \mathfrak{\bm{C}}=\widehat{\bm{C}}\bm{X}$. Let  $m_0=w_0$ and $m_l=\mathfrak{B}_l\mathfrak{C}_l$ with $l=1,2,\cdots,P-1$.
\State \Return $\{m_l\}_{l=0}^{P-1}$ and $\{s_l\}_{l=0}^{P-1}$.
\end{algorithmic}
\end{algorithm}

\begin{remark}
Since the VPMR approach requires high-precision matrix manipulation, we employ the Multiple Precision Toolbox \cite{MP}
in order to implement the algorithm. These packages are used in both steps of the VP-sum and the model-reduction procedures.
The computer code of the VPMR approach is released as open source, which is available at https://github.com/ZXGao97.
\end{remark}

\begin{remark}\label{remark1}
	If the smooth function $f(x)$ has a limit at infinity, our method works on the whole positive-axis, thus the interval could be an arbitrary subset of $\mathbb{R}$. If the limit does not exist, the above VPMR method has low accuracy because of the discontinuity of the transformed function $K(r)$. This reason will be further explained in section \ref{errorestimate}. In order to achieve higher
	accuracy, one could truncate $f(x)$ at a required point, then connect a fast decreasing function (for example, the Gaussian function) behind the cutoff point such that a new function $f^*(x)$ to localize the kernel function. The higher-order differentiable properties of $f^*(x)$ can be kept.
\end{remark}

\begin{remark}
After the VP approach, the exponents $-j/n_{c}$ and the weights $w_{j}$ given in the left hand of Eq.\eqref{mdrec} are real, whereas the MR approach does not guarantee that the exponents $-s_{\ell}$ and the weights $m_{\ell}$ given in the right hand of Eq.\eqref{mdrec} are real. In most cases, $-s_{\ell}$ and $m_{\ell}$ are complex whereas the MR approach	guarantees that the real parts of $s_{\ell}$ are positive for all $\ell$ \cite{Moore1981Principal}. An interesting theorem \cite{liu1998model} is that $s_{\ell}$ are real and $m_{\ell}$ are positive when $w_{j}$ are positive. It seems to be more useful for other integral discretization-based SOE approximations whose weights are all positive \cite{Beylkin2005Approximation}. However, the limitation of such integral-based method is the difficulty in extending to general kernels.
\end{remark}

\subsection{Computation details and complexity analysis of the VPMR algorithm}
In this subsection, we perform some analysis to the complexity of the VPMR method.

First, we consider the VP-sum procedure. The coefficients $a_j$ in Eq.\eqref{coefficients} are evaluated using either the adaptive Gauss-Legendre integral quadrature \cite{shampine2008vectorized} (function ``quadgk'' in MATLAB) or the 2D dilation quadrature (degenerate into 1D case) \cite{occorsio2018cubature}, with average number of quadratures $\tilde{N}$. The adaptive Gauss-Legendre quadrature is efficient for kernel with non-oscillation or low frequency oscillation, whereas may not work for high frequency oscillating integral. The 2D dilation quadrature is more appropriate for integrating near-singular and oscillating kernels with high precision and efficiency. The construction of the coefficients $w_j$ using Eq.\eqref{Expansion-coefficients} requires expensive combinatorial number, but can be done via simple combinatorial recursions
\begin{equation}
	{\binom{n+\ell+j}{n+\ell-j}}=\dfrac{(2j+1)(n+\ell+j)}{(n+\ell-j)(n+\ell-j-1)}{\binom{n+\ell+j-1}{n+\ell-j+1}}.
\end{equation}
The complexity of $w_j$ is thus $\sim O(2n^2)$. The total complexity of the procedure of the VP-sum is $O(2n\tilde{N}+2n^2)$.

Second, we study the MR procedure based on square root factorization as given in Algorithm \ref{modelreduction}. Because $\bm{A}$ is a $(2n-1)\times(2n-1)$ diagonal matrix, $\bm{b}$ and $\bm{c}$ are column and row vectors of dimension $(2n-1)$, the two Lyapunov equations can be solved directly via $\sim O((2n-1)^2)$ operations due to the productions of $\bm{b}\bm{b}^{*}$ and $\bm{c}^{*}\bm{c}$. To obtain the Hankel singular value, we compute the Cholesky factor $\bm{S}$ and $\bm{L}$ of the solutions of the Lyapunov equations $\bm{P}$ and $\bm{Q}$ (both $\bm{P}$ and $\bm{Q}$ are symmetric positive definite), respectively, and then compute the singular value decomposition (SVD) of $\bm{S}^{T}\bm{L}$. The efficiently implemented Cholesky factorization requires $(2n-1)^3/6+O((2n-1)^2)$ operations \cite{krishnamoorthy2013matrix}. The SVD factorization and the matrix-matrix products in the Step 4-6 of Algorithm \ref{modelreduction} need $O((2n-1)^3)$ operations. In the Step 7-10 of Algorithm \ref{modelreduction}, we construct a low-rank approximation of the whole system by truncating the Hankel singular value at $k$-th. Similar analysis can be applied that the complexity for Step 7-10 yields $O(k^3)$.

Overall, the total computational complexity of the VPMR method is given by $O(2n\tilde{N}+2n^2+(2n-1)^3/6+(2n-1)^2+(2n-1)^3+k^3)$. The leading part is contributed from the SVD factorization appeared in Step 4 of the Algorithm \ref{modelreduction}. We note that $n\approx 1000$ is generally accurate enough for the VPMR method in most cases, thus the complexity is relatively not very big though $O((2n)^3)$ operations are needed. Furthermore, some approaches can be used for improving the efficiency like the randomized singular value decomposition (RSVD) method \cite{halko2011finding,rokhlin2010randomized} which reduce the computational cost of SVD from $O(n^3)$ to $O(kn^2+k^2n)$.

\subsection{Error estimate of the VPMR algorithm}\label{errorestimate}
In this subsection, we discuss the convergence rate of the VPMR method with respect to the number of terms $n$ to approximate the kernel function $f(x)$. Because the errors of the VP-sum and the MR approach can be treated separately, we only need to do the error estimate of approximating $K(r)$ using the VP-sum $V_n[K(r)]$ in Eq.\eqref{Vn}, due to the estimate of the MR approach has been given in Eq.\eqref{MRestimate}.

An observation is that the $2\pi$-periodic function $K(r)$ is smooth on $(0,\pi)$ and $(-\pi,0)$ whereas may not become $C^{\infty}$ at $r\in\{-\pi,0,\pi\}$, i.e., $K(r)$ is piecewise smooth at its definitional domain. This is because that we make an even and periodic prolongation of $K(r)$ such that $K(r)$ is defined on the whole axis $(-\infty,\infty)$ instead of $[0,\pi]$. Approximation properties of the VP-sum in the uniform metric are considered for certain classes of continuous and smooth functions in many papers \cite{lebesgue1909integrales,zakharov1968bound,stechkin1961approximation,efimov1959approximation,telyakovskii1958approximation,nikolski1940certaines}. However, local approximation properties of the VP-sum for piecewise smooth functions have been little studied. From recent work \cite{2020kernel}, the results are summarized as the following theorem.
\begin{theorem}\label{theorem2}
	If $K(r)$ defined by the variable change Eq.\eqref{variable} is not differentiable at $r=0$, its VP-sum  $V_n[K(r)]$ still converges to $K(r)$ uniformly on $\mathbb{R}$, whereas the rate of convergence at $r = 0$ is much slower than that at any other point, that
	\begin{equation}
		\begin{split}
		&V_n[K(0)]-K(0)=-\frac{\ln2}{n\pi}K'(r)\big{|}_{r=0}+o(n^{-1}),\\
		&V_n[K(r)]-K(r)=o(n^{-1})~\text{for}~r\neq0.
		\end{split}
	\end{equation}
	Furthermore, suppose that $K(r)$ is twice-differentiable on $[0,\pi]$ with $K'(r)\big{|}_{r=0}=0$ and period $2\pi$, then we have
	\begin{equation}
		\begin{split}
		&V_n[K(0)]-K(0)=O(n^{-2}),\\
		&V_n[K(r)]-K(r)=o(n^{-2})~\text{for}~r\neq0.
		\end{split}
	\end{equation}
\end{theorem}
In fact, one can simply check that the function $K(r)$ defined by the variable change Eq.\eqref{variable} is naturally satisfies $K'(r)\big{|}_{r=0}=0$. This result means that our VP-sum approximation has at least two order convergence rate for independent kernels.

Furthermore, the limitations of above error estimate Theorem \ref{theorem2} are twofold. First, for the kernel with higher order continuity conditions, the VP-sum may have higher convergence rate (for example, exponential decay when $K(r)$ becomes even function and vanishes at $r=\pi$) but not provided. Second, if $K(r)$ itself is a piecewise smooth function (for example, when $f(x)$ is required to localized as in Remark \ref{remark1}), Theorem \ref{theorem2} is not valid. For these concerns, we try to derive a more general relationship between $K(r)$ and its VP-sum as follows.  To prove this relationship, we need to use the following lemma.

    \begin{lemma}\label{lamma21}
	The following inequalities hold for $n\geq 1$
	\begin{equation}
		\mathlarger{\int}_{-\pi}^{\pi}\left|n^{\eta}\sum_{\ell=n}^{2n-1}\sum_{j=\ell+1}^{\infty}\dfrac{\sin\left(j\tau-\dfrac{\eta\pi}{2}\right)}{j^{\eta+1}}\right|d\tau\leq
		\begin{cases}
			&\dfrac{16(\eta+2)^2}{\left(1+\frac{1}{n}\right)^{\eta}},~\eta\geq 1,\\\\
			&32\ln\left(1+\dfrac{n}{n+1}\right)+58,~\eta=0.
		\end{cases}	
	\end{equation}
\end{lemma}
The proof of Lemma \ref{lamma21} is an extension of Lemma 2.3 from Ref.~\cite{sharapudinov2012approximation} and Statement 2 from Ref.~\cite{magomed2016approximation}, we omit it.

\begin{theorem}\label{theorem3}(The error bound of the VP sum)
	Denote $W^{\eta,m}(\Omega)$ the Sobolev spaces which satisfies
	\begin{equation}
		W^{\eta,m}(\Omega)=\{u\in L^{m}(\Omega)~|~\tilde{\partial}^{\alpha}u\in L^{m}(\Omega),|a|\leq \eta\},
	\end{equation}
    where $\tilde{\partial}$ indicates the generalized derivatives operation.
	Let $\digamma=\{\mathcal{T}_j,j=0,\cdots,\mathcal{P}\}$ becomes a finite division of the interval $[-\pi,\pi]$ which satisfies
	\begin{equation}
		-\pi=\mathcal{T}_0<\mathcal{T}_1<\cdots<\mathcal{T}_{\mathcal{P}}=\pi
	\end{equation}
such that the $2\pi$-periodic function $K(r)\in W^{\eta,\infty}([-\pi,\pi])$ can be converted into a function from $C^{\infty}([\mathcal{T}_i,\mathcal{T}_{i+1}])$ by redefining it at the endpoints on each closed interval $[\mathcal{T}_i,\mathcal{T}_{i+1}]$. Then the following estimate of the remainder term of the VP-sum holds:
\begin{equation}
	\left|V_n[K(r)]-K(r)\right|\leq	\begin{cases}
		\dfrac{16(\eta+2)^2M_{\eta+1}}{\pi n (n+1)^{\eta}},\,~\text{for} ~ \eta\geq1,\\\\
	    \dfrac{M_{\eta+1}}{n\pi}(58+32\ln 2),\,~\text{for}~\eta=0,
	\end{cases}
\end{equation}
where $M_{\eta+1}=\max\limits_r\left\{\left|K^{(\eta+1)}(r)\right|\right\}$.
\end{theorem}

\begin{proof}
	Recalling the definition of $n$-th VP-sum Eq.\eqref{Vn} with $n\geq 1$, the approximate error is given by
	\begin{equation}\label{divide}
     		V_n[K(r)]-K(r)=-\dfrac{1}{n\pi}\sum_{\ell=n}^{2n-1}\sum_{j=\ell+1}^{\infty}\cos(jr)\int_{-\pi}^{\pi}K(\tau)\cos(j\tau)d\tau.
	\end{equation}
    Using the division $\digamma$ defined on $[-\pi,\pi]$, we rewrite the integral in the right term of Eq.\eqref{divide} as the sum of the integrals over the integral domain formed by $\digamma$.

    \begin{equation}\label{splitting2}
    	\int_{-\pi}^{\pi}K(\tau)\cos(j\tau)d\tau=\sum_{i=1}^{\mathcal{P}}\int_{\mathcal{T}_{i-1}}^{\mathcal{T}_i}K(\tau)\cos(j\tau)d\tau.
    \end{equation}
     After $\eta+1$ times integrating by parts, we rewrite the integration at each subintegral in the form
     \begin{equation}\label{integralbypart}
     	\begin{split}
     	   \int_{\mathcal{T}_{i-1}}^{\mathcal{T}_i}K(\tau)\cos(j\tau)d\tau
     	   =&\sum_{k=0}^{\eta}\dfrac{K^{(k)}_{-}(\mathcal{T}_i)\sin\left(j\mathcal{T}_{i}-\dfrac{k\pi}{2}\right)-K^{(k)}_{+}(\mathcal{T}_{i-1})\sin\left(j\mathcal{T}_{i-1}-\dfrac{k\pi}{2}\right)}{j^{k+1}}\\
     	   &+\mathlarger{\int}_{\mathcal{T}_{i-1}}^{\mathcal{T}_i}K^{(\eta+1)}(\tau)\dfrac{\sin\left(j\tau-\dfrac{\eta\pi}{2}\right)}{j^{\eta+1}}d\tau,
     	\end{split}
     \end{equation}
 where the subscripts $-$ and $+$ indicate the left sided limit and the right sided limit, respectively. Substituting the expressions Eq.\eqref{integralbypart} and Eq.\eqref{splitting2} into the remainder term Eq.\eqref{divide} and recalling $K(r)$ is a $2\pi$-periodic function, therefore,
 \begin{equation}\label{estimate1}
 	\begin{split}
 	   V_n[K(r)]-K(r)=&-\dfrac{1}{n\pi}\sum_{\ell=n}^{2n-1}\sum_{j=\ell+1}^{\infty}\cos(jr)\sum_{i=1}^{\mathcal{P}}\sum_{k=0}^{\eta}\bigggg{[}\dfrac{K^{(k)}_{-}(\mathcal{T}_i)\sin\left(j\mathcal{T}_{i}-\dfrac{k\pi}{2}\right)}{j^{k+1}}\\
 	   &-\dfrac{K^{(k)}_{+}(\mathcal{T}_{i-1})\sin\left(j\mathcal{T}_{i-1}-\dfrac{k\pi}{2}\right)}{j^{k+1}}\bigggg{]}\\
 	   &-\dfrac{1}{n\pi}\sum_{\ell=n}^{2n-1}\sum_{j=\ell+1}^{\infty}\cos(jr)\sum_{i=1}^{\mathcal{P}}\mathlarger{\int}_{\mathcal{T}_{i-1}}^{\mathcal{T}_i}K^{(\eta+1)}(\tau)\dfrac{\sin\left(j\tau-\dfrac{\eta\pi}{2}\right)}{j^{\eta+1}}d\tau.
 	\end{split}
 \end{equation}
Note that the first term in Eq.\eqref{estimate1} vanishes because $K(r)\in W^{\eta,\infty}([-\pi,\pi])$ such that the left-hand limits and right-hand limits of $\{K^{(k)}(r),k=0,1,\cdots,\eta\}$ at the endpoint in the division $\digamma$ are equal. Taking Lemma \ref{lamma21} into Eq.\eqref{estimate1}, for $\eta\geq1$, we obtain
\begin{equation}\label{estimate5}
	\begin{split}
		 \left|V_n[K(r)]-K(r)\right|&=\left|\dfrac{1}{n\pi}\sum_{\ell=n}^{2n-1}\sum_{j=\ell+1}^{\infty}\cos(jr)\sum_{i=1}^{\mathcal{P}}\mathlarger{\int}_{\mathcal{T}_{i-1}}^{\mathcal{T}_i}K^{(\eta+1)}(\tau)\dfrac{\sin\left(j\tau-\dfrac{\eta\pi}{2}\right)}{j^{\eta+1}}d\tau\right|\\
		&\leq \dfrac{M_{\eta+1}}{n^{\eta+1}\pi}\mathlarger{\int}_{-\pi}^{\pi}\left|n^{\eta}\sum_{\ell=n}^{2n-1}\sum_{j=\ell+1}^{\infty}\dfrac{\sin\left(j\tau-\dfrac{\eta\pi}{2}\right)}{j^{\eta+1}}\right|d\tau\\
		&\leq\dfrac{16(\eta+2)^2M_{\eta+1}}{\pi n (n+1)^{\eta}}.
	\end{split}
\end{equation}
Similarly,
\begin{equation}
	\begin{split}
	\left|V_n[K(r)]-K(r)\right|&\leq \dfrac{M_1}{n\pi}\left[32\ln\left(1+\frac{n}{n+1}\right)+58\right]\\
	&\leq \dfrac{M_1}{n\pi}(32\ln2+58)
	\end{split}
\end{equation}
when $\eta=0$.
\end{proof}

Theorem \ref{theorem3} states that if $K(r)$ is a smooth function except the origin, and is $\eta$ times continuously differentiable at $r=0$, and $K^{(\eta+1)}(r)$ bounded, then the VP-sum of $K(r)$ has at least $(\eta+1)$-order convergence rate. For example, the Gaussian kernel $f(x)=e^{-x^2/4\delta}$ is an even function that the corresponding $K(r)$ is infinitely differentiable such that the VP-sum of the Gaussian kernel has spectral convergence. For general cases, assuming that the approximate interval we interested in is $[0,\mathfrak{F}]$ and a fast decreasing function may connect after $\mathfrak{F}$ to localize the kernel function as in Remark \ref{remark1}, then the following convergence rate of the VP-sum of the new function $f^{*}(x)$ at $[0,\mathfrak{F}]$ holds
\begin{theorem}\label{backtof}
	 Assume that $f^{*}(x)$ is a smooth function defined on the positive axis except the connecting point $x=\mathfrak{F}$, and fast decrease at $x\rightarrow\infty$. After the variable change and the even and periodic prolongation, if $f^{*}(x)$ has bounded derivatives at $x=\mathfrak{F}$ until $\eta$-th order, then the convergence rate of the VP-sum of the resulting function $K(r)$ is at least $\mathcal{O}(n^{-\eta-1})$.
\end{theorem}
By Ref.\cite{huang2006chain}, Theorem \ref{backtof} can be derived by the following theorem, which is a frequently-used formula in obtaining higher derivatives.
	\begin{theorem}
		(Fa{\`a} di Bruno's Formula) If $f$ and $\mathfrak{S}$ are functions with a sufficient number of derivatives, then
		\begin{equation}
			\frac{d^{m}}{d r^{m}} f(\mathfrak{S}(r))=\sum \frac{m !}{b_{1} ! b_{2} ! \cdots b_{m} !} f^{(k)}(\mathfrak{S}(r))\left(\frac{\mathfrak{S}^{\prime}(r)}{1 !}\right)^{b_{1}}\left(\frac{\mathfrak{S}^{\prime \prime}(r)}{2 !}\right)^{b_{2}} \cdots\left(\frac{\mathfrak{S}^{(m)}(r)}{m !}\right)^{b_{m}}
		\end{equation}
	where the sum is over all different solutions in nonnegative integers $b_1,\cdots,b_m$ of $b_1+2b_2+\cdots+mb_m=m$, and $k:=b_1+\cdots+b_m$.
	\end{theorem}
Now back to the proof of Theorem \ref{backtof}.

\begin{proof}[Proof of Theorem \ref{backtof}]
Recalling the variable change given in Eq.\eqref{variable}, the corresponding $\mathfrak{S}(r)$ reads
\begin{equation}
	\mathfrak{S}(r)=-n_c\log\left(\dfrac{1+\cos r}{2}\right),~r\in[0,\pi].
\end{equation}
It is observed that $\mathfrak{S}(r)$ is an even function such that $\mathfrak{S}^{(k)}(r)$ vanishes and $\mathfrak{S}^{(k)}(r)\sim{O}(r)$ as $r\rightarrow0$ when $k$ is an odd number. When $k$ is an even number, it can be checked that $K^{(k)}_{-}(0)=K^{(k)}_{+}(0)=\partial^{(k)}f(\mathfrak{S}(0))$. Above all, $K(r)$ is infinitely differentiable at the origin, i.e., $K(r)\in W^{\eta,\infty}[-\pi,\pi]$ due to the localization at $\mathfrak{S}^{(k)}(r)=\mathfrak{F}$. After the use of Theorem \ref{theorem3}, we finish the proof.
\end{proof}

\subsection{SOE approximations of frequently-used kernels}
\subsubsection{The Gaussian kernel}
We investigate the performance of the VPMR approach for the SOE approximations of the Gaussian kernel
which is frequently required in many applications \cite{jiang2019fast,SpivakThe}.
The inverse Laplace transform representation of the Gaussian kernel $e^{- x^2/4\delta}$ is,
\begin{equation}
e^{-\frac{x^2}{4\delta}}=\frac{1}{2\pi i}\int_{\Gamma}e^{z}\sqrt{\frac{\pi}{z}}e^{-\frac{\sqrt{z}|x|}{\sqrt{\delta}}}dz,
\end{equation}
where $\Gamma$ can be any contour in the complex plane that starts from $-\infty$ in the third quadrant, goes around
$0$ and returns back to $-\infty$ in the second quadrant. There are mainly three kinds of contours, including the parabolic,
the hyperbolic and the modified Talbot contours \cite{lopez2006spectral,trefethen2006talbot}. All these contours have
certain parameters that need to be optimized in order to achieve optimal convergence rate \cite{talbot1979accurate,weideman2007parabolic,weideman2010improved}.
Alternatively, one can obtain an SOE approximation of Gaussian kernel by the best rational approximation using the residue theorem and Cauchy's theorem \cite{trefethen2006talbot}.

We perform the comparison on SOE approximations of the Gaussian kernel using our VPMR method given in Section \ref{VPMR} and
other existing work discussed above. We take $\delta=1$. The maximum error
\begin{equation}\label{mearsurement}
E_{\infty}=\max\limits_{x\in (0,100]}\left|e^{-\frac{x^2}{4\delta}}-\sum_{j}m_je^{-s_jx}\right|
\end{equation}
is used to measure the performance, where $100000$ monitoring points are randomly sampled from $[10^{-5},10^2]$
to estimate the maximum. For the comparison,
the results of contour integrals and the best rational approximations are taken
from Jiang \cite{jiang2019fast}, where the methods are fully optimized and the contours are discretized via the midpoint rule.
And the MR is used to reduce the number of exponentials.
For the VPMR, the parameter $n_c$ is set to be $\lceil n/4 \rceil$, thus the maximum exponent is about $8$.
The results are displayed in Figure \ref{ComparisonGaussian}. The results demonstrate that the VPMR has advantages
in both convergence rate and accuracy.
For all five methods, an error level of $10^{-13}$ can be achieved with the reduced number $P$ of exponentials in the SOE approximation not bigger than
20.  It was measured \cite{jiang2019fast} that the convergence rate
is about $O(6.3^{-n})$ of all three contours, and $O(7.5^{-n})$ for the best rational approximation. The VPMR achieves a convergence rate of about $O(9.0^{-n})$.

\begin{figure}[tb]
\centering
	\includegraphics[width=0.60\textwidth]{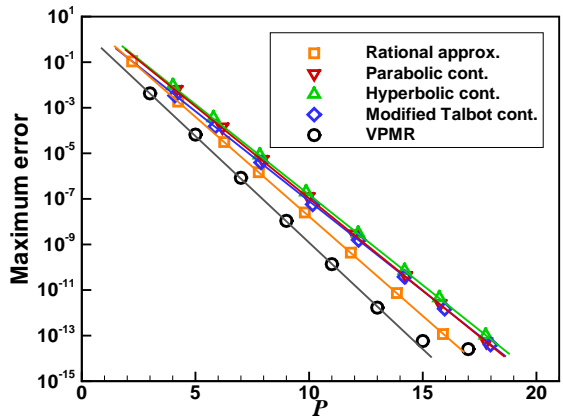}
\caption{Maximum errors of the SOE approximations of the Gaussian kernel with the number of exponentials.
Data are shown for five SOE methods:  the best rational approximation, the parabolic contour,
the hyperbolic contour, the modified Talbot contour and the VPMR. The dash-dotted lines with different color indicate the fitting lines
of corresponding SOE methods.}
\label{ComparisonGaussian}
\end{figure}

\subsubsection{SOE approximations of other important kernels}
We study the performance of the SOE to approximate the exact kernels with the increase of $P$. Four different kernels, the Mat$\acute{\text{e}}$rn kernel, the power function, the Ewald splitting kernel, and the Helmholtz kernel, are used to test the algorithm. Unlike the Gaussian kernel, some of these kernels have complicated forms thus are difficult to obtain SOE via theoretical approach. We perform the comparison on SOE approximations using the classical Prony's method \cite{hamming2012numerical} which requires only the value of kernel at discrete points. We use the Prony Toolbox \cite{PronyToolbox} which is a software tool in MATLAB in order to perform the Prony analysis.

The first one is the Mat$\acute{\text{e}}$rn kernel \cite{RogerInterpolation,Williams2005Gaussian,Alexander2018Gaussian,Dral2019MLatom}, often used as a covariance function in modeling Gaussian processes and machine learning. The Mat$\acute{\text{e}}$rn kernel of order $\nu>0$ is defined as
\begin{equation}\label{MaternKernel}
f(x)=\dfrac{(\sqrt{2\nu}|x|)^{\nu} K_\nu(\sqrt{2\nu}|x|)}{2^{\nu-1}\Gamma(\nu)},
\end{equation}
where $\nu$ is the smoothness parameter, $K_\nu$ is the modified Bessel function of the second kind of order $\nu$ and $\Gamma$ is the Gamma function. The definition Eq.\eqref{MaternKernel} satisfies that $f(0)=1$ for any $\nu$ and $f(x)$ has high-order differentiability at the origin when $\nu$ is larger. In this example, we employ VPMR method to perform the comparison on SOE approximations of the Mat\'{e}rn kernel with different smoothness $\nu$ and $\delta=0$. Numerical results are given in Figure \ref{KernelFunction} (a), which demonstrate that the convergence rate of the SOE becomes higher for better smoothness. The efficiency of the resulting SOE produced by the VPMR is very attractive and it needs less than $40$ terms for all $\nu$ to achieve $10^{-9}$ maximum error, whereas the convergence rate of the Prony method is slow and only $10^{-3}$ for all $\nu$ can be achieved.

\begin{figure}[!h]
	\centering
	\includegraphics[width=0.45\textwidth]{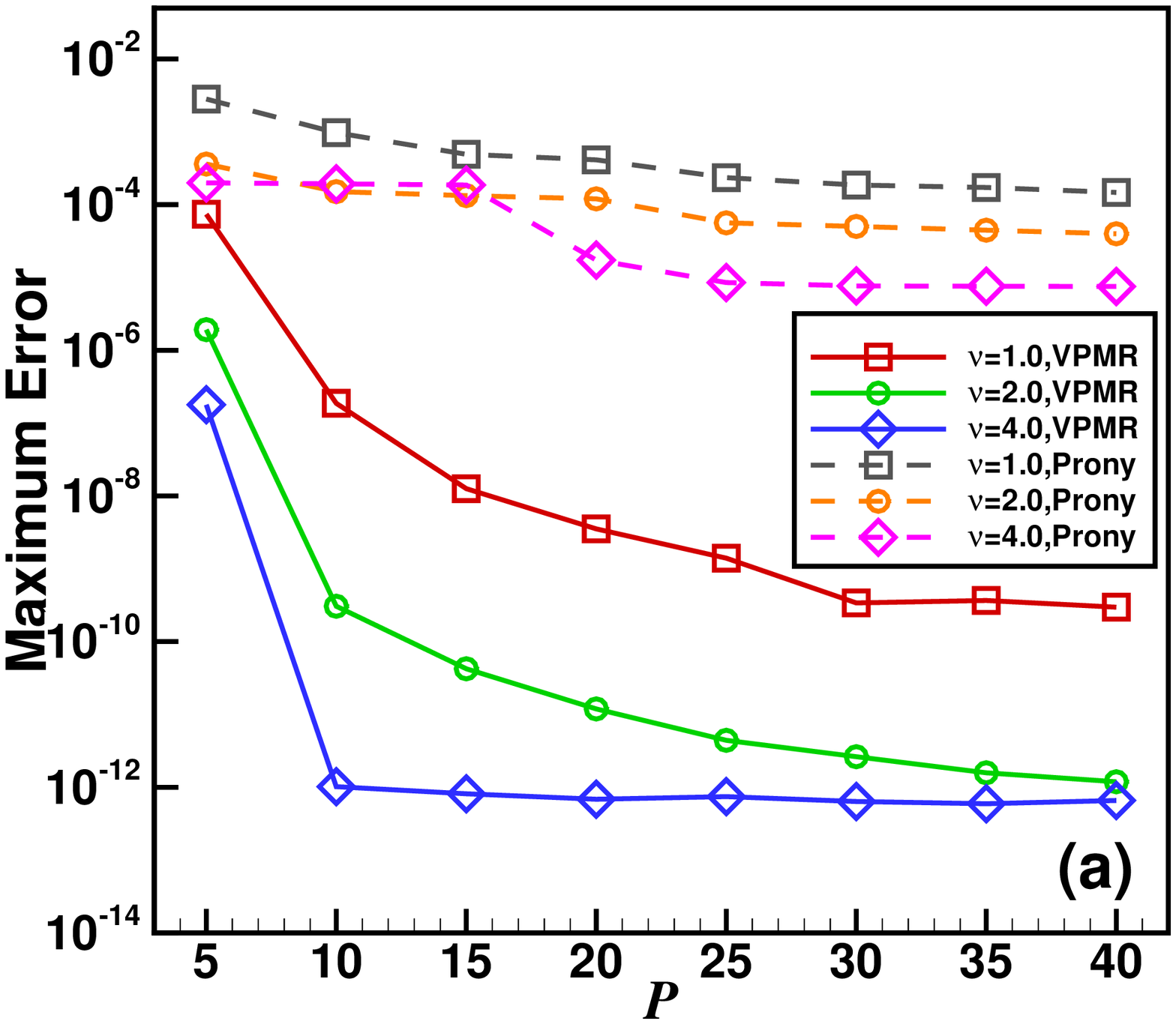}
	\includegraphics[width=0.45\textwidth]{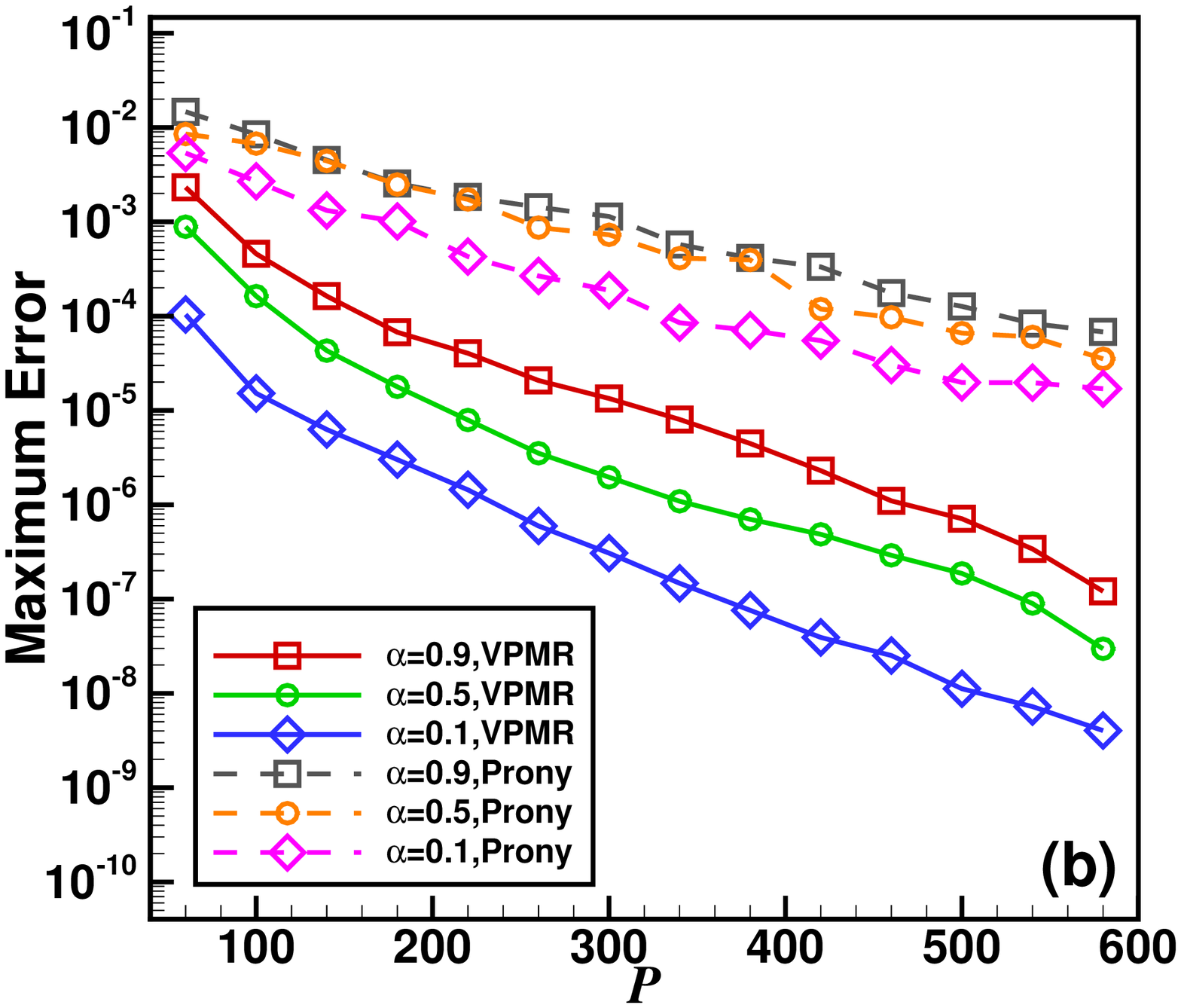}
	\includegraphics[width=0.45\textwidth]{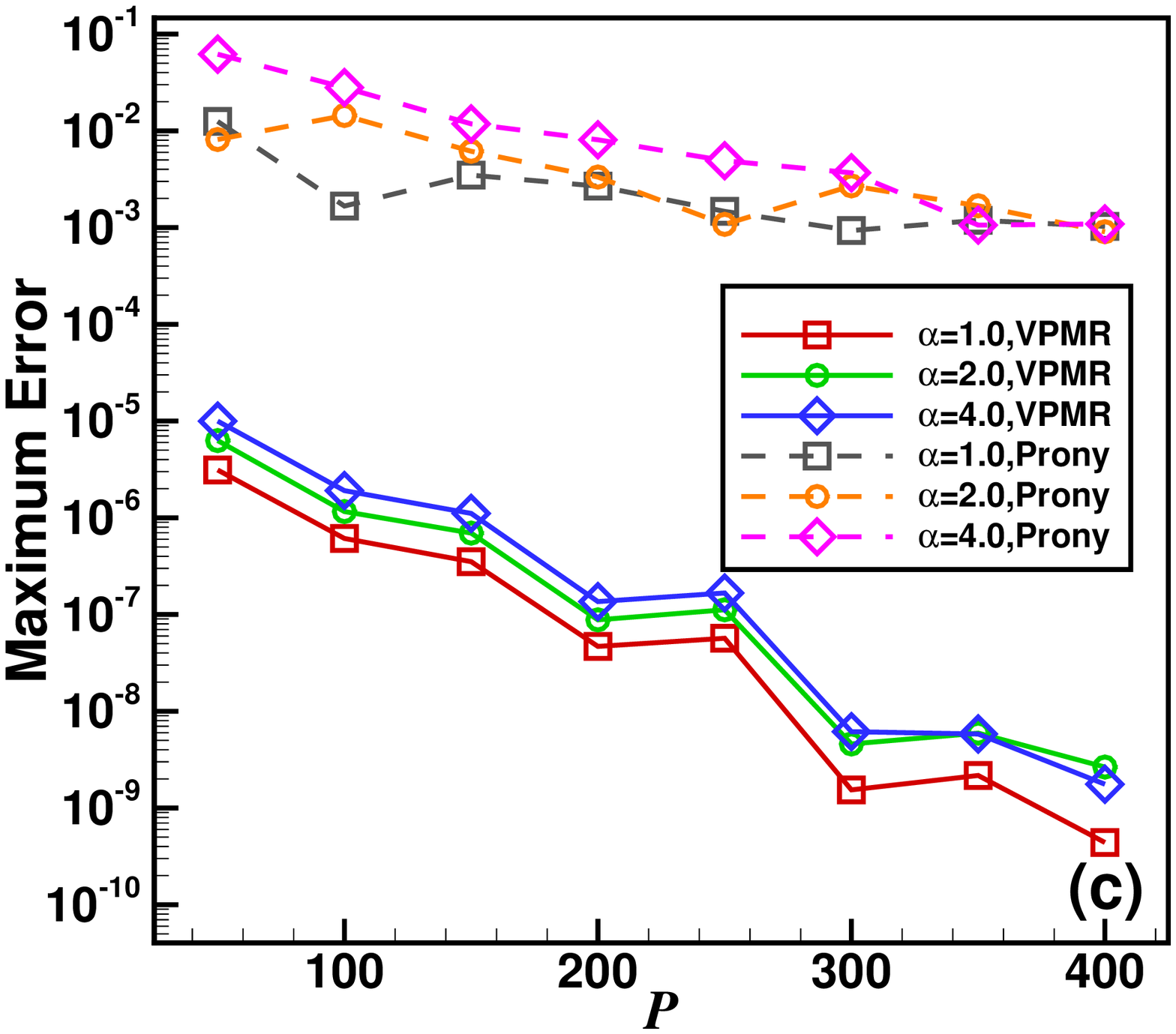}
	\includegraphics[width=0.45\textwidth]{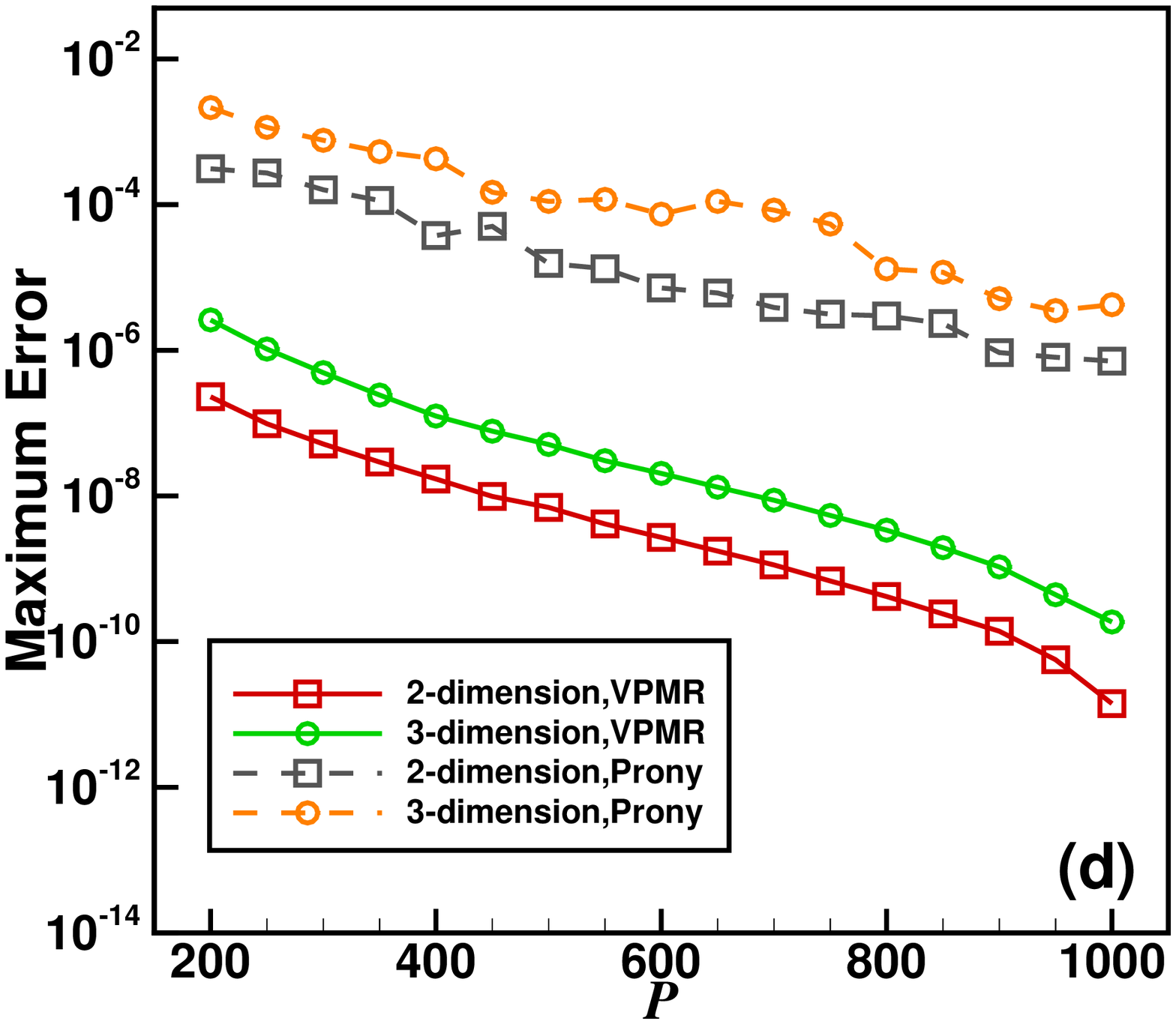}
	\caption{Maximum errors of the SOE approximations of different kernels, on comparison of the VPMR with the Prony, (a) the Mat\'{e}rn kernel, (b) the power function kernel, (c) the Ewald splitting kernel, and (d) Helmholtz kernel, as fucntions of the number of exponentials.}
	\label{KernelFunction}
\end{figure}

Secondly, we consider the weak-singular power function with $f(x)=x^{\alpha-1}$, which is frequently required in computational physics \cite{Yong2018The,Beylkin2005Approximation}. The power function has the following inverse Laplace transform expression \cite{Beylkin2010Approximation},
\begin{equation}\label{powerLaplace}
	x^{\alpha-1}=\dfrac{1}{\Gamma(1-\alpha)}\int_{-\infty}^{\infty}e^{-e^{t}x+(1-\alpha)t}dt.
\end{equation}
A suitable quadrature rule to Eq.\eqref{powerLaplace} yields an explicit discretization to obtain a sum of exponentials. Although constructing the SOE approximation from integral representation can achieve a given accuracy, the uncontrollable increase of exponent values limits the practical use (as the convergence conditions require a tolerable exponent value, see Ref.\cite{Lubich1993Runge} and Theorem \ref{theorem1} in this paper). The results on SOE approximation of the power function using the VPMR method are given in Figure \ref{KernelFunction} (b), for three different values of $\alpha$. The value of $\delta$ is set to $0.05$ and the maximal exponent is set to about $5$. Its needs $600$ terms for all $\alpha$ to achieve $10^{-8}$ maximum error. This convergence is fast considering that the SOE based on the inverse Laplace transform requires the maximal exponent to be about $10^3$ to achieve the same level of accuracy. Furthermore, comparing with the results conducted by the Prony method, the convergence rate of the VPMR method is much more fast.

When $\alpha=0$, the singular power function with expression $1/x$ is referred as the Green's function of 3D Poisson's equation with free boundary condition, which is frequently used in fast electrostatic sums. In practice, it is not straightforward to evaluate the electrostatic interaction because $1/x$ has properties of both short range and long range. The Ewald method solves this problem by dividing the Coulomb potential into near and far parts \cite{ewald1921berechnung},
\begin{equation}
	\frac{1}{x}=\frac{\text{erfc}\left(\Lambda x\right)}{x}+\frac{\text{erf}\left(\Lambda x\right)}{x},
\end{equation}
where the parameter $\Lambda$ describes the inverse of cutoff radius and balances the proportion of computational cost between these two parts. The near part is often truncated at a suitable cutoff radius and computed as an explicit sum of pairwise interactions. Here we study the SOE for the far part Ewald kernel $\text{erf}(\Lambda x)/x$, which is often studied in Fourier space \cite{JinLiXuZhao2020},
for different values of $\Lambda$.  Figure \ref{KernelFunction} (c) shows the errors of the SOE approximation as a function of $P$ for different $\Lambda$. We observe that the VPMR needs $400$ terms for all $\Lambda$ to achieve $10^{-9}$ maximum error, whereas the Prony has about $10^{-3}$ maximum error with the same number of terms which has six orders of magnitude lower than the VPMR.

Finally, we test the performance of the VPMR for the Helmholtz kernel which has a strong oscillation. The kernel is the Green's function of the Helmholtz equation. In 2D and 3D, they are given by,
\begin{equation}\label{2dgreen}
    f_{\text{2D}}(x)=\dfrac{i}{4}H_0^{(1)}(kx),~~~~f_{\text{3D}}(x)=\dfrac{e^{ikx}}{4\pi x},
\end{equation}
where $H_0^{(1)}$ is the zeroth order Hankel function of the first kind. The Helmholtz equation with high wave number is notoriously difficult to solve numerically, since the larger the value of $k$, the stronger the oscillation of the Helmoltz kernel.
An efficient SOE approximation for Helmholtz kernel is also difficult due to this issue. We study the performance of the VPMR approach for the SOE approximations of 2D and 3D Helmholtz kernels on domain $[\delta,10]$ with $\delta\ll1$. The results are given in Figure \ref{KernelFunction} (d) with parameters $k=50$ (a large wave number) and $\delta=0.05$. One can observe that $1000$ exponentials achieve the error levels of $10^{-9}$  and $10^{-10}$ for 2D and 3D cases, respectively, whereas the Prony achieves the lower levels of $10^{-5}$ for both cases. Note that the 2D/3D Helmholtz kernel yields the well known Sommerfeld integral representation
\begin{equation}\label{Sommerfeld}
	f_{2D}(\bm{r})=\dfrac{1}{4\pi}\int_{-\infty}^{\infty} \dfrac{e^{-\sqrt{\lambda^{2}-k^{2}}\left|r_y\right|}}{\sqrt{\lambda^{2}-k^{2}}} e^{i \lambda r_x} d \lambda
\end{equation}
and
\begin{equation}\label{Sommerfeld3D}
	f_{3D}(\bm{r})=i\int_{0}^{\infty} \dfrac{k_{\rho}}{k_z}J_0(k_{\rho}r_{\rho})e^{ik_z|r_z|}dk_{\rho},
\end{equation} 	
where $\bm{r}=(r_x,r_y)$ for 2D, $\bm{r}=(r_{\rho}, r_{\theta}, r_z)$ for 3D, $k_z=\sqrt{k^2-k_{\rho}^2}$, and $J_0$ is the zeroth order Bessel function of the first kind. The Sommerfeld integrals Eq.\eqref{Sommerfeld} and Eq.\eqref{Sommerfeld3D} play a central role in scattering theory. Numerical calculation of the Sommerfeld integration requires contour deformation to avoid the square
root singularity in the integrand and additional approaches to deal with the probable highly oscillatory integrand. The SOE approximation provided by the VPMR method may furnish an useful tool to design fast algorithm for wide applications in electromagnetic scattering field, which will be explored in our future work.

\section{The Application of SOE on Convolution Quadrature}\label{fastconv}
One of important applications of the SOE is to quickly approximate convolution quadrature. In this section, we consider the approximation of convolution
quadrature between a given kernel $f(t)$ and smooth function $g(t)$ as follows,
\begin{equation}\label{TimeConv}
y(t)=f*g=\int_0^tf(t-\tau)g(\tau)d\tau.
\end{equation}
The convolution quadrature approximation has attracted wide interest due to its
 broad applications in scientific computing such as integro-partial differential equations \cite{lopez1990difference,sanz1988numerical}, fractional differential equations \cite{CAO2013154,cuesta2003fractional,diethelm2002analysis,banjai2019efficient,li2010fast}, and nonlinear Volterra equations \cite{ZAKERI20106548,lubich1985fractional,miller1975volterra,MOHAMMADI2015254}.
The temporal convolution quadrature is often evaluated based on a technique of the Laplace inverse transform
and Runge-Kutta (RK) time stepping, proposed by Lubich \cite{lubich1988convolution,lubich1988convolution2}.
As introduced in the introduction part, with the Laplace transform $F(s)$ of the kernel function, the convolution integral can be written as,
\begin{equation}
y(t)=\dfrac{1}{2\pi i}\int_\Gamma F(\lambda)\int_0^t e^{\lambda(t-\tau)}g(\tau)d\tau d\lambda.
\end{equation}
Let $u(t)=\int_0^t e^{\lambda(t-\tau)}g(\tau)d\tau$, which satisfies ODE
$u'=\lambda u+g(t)$ with $u(0)=0$ and can be integrated by the RK method. The contour integral on $\Gamma$
can be calculated by employing a numerical quadrature. This method is essentially an interpolation
method which uses the linear combination
of the interpolation points of function $g(t)$ for approximating the convolution.
The attractive features of convolution quadratures include that they work well for kernels of singular,
multiple time scales, and highly oscillatory with given the  Laplace transform of the kernel function \cite{lubich2004convolution}.
Sch\"adle {\it et al.} \cite{schadle2006fast} developed an improved algorithm with $O(N \log N)$ multiplications and
$O(\log N)$ active memory for $N$ time steps. L\'opez-Fern\'andez and Sauter \cite{lopez2013generalized}
introduced a generalized convolution quadrature allowing for variable time steps.
These improved algorithms make the Lubich's method more adaptable and less storage space.
However, it was pointed out that this type of algorithms is restricted to
sectorial convolution kernels and thus not applicable to wave equations \cite{lopez2013generalized}. Additionally, it requires
the analytical Laplace transform of the kernel function, which can be difficult for some functions
such as the Mat\'{e}rn kernel often used in machine learning and statistics \cite{march2015kernel}. One has
to approximate the kernel by rational polynomials to obtain the Laplace transform, leading to a
difficulty of error estimates.

Lubich's method is equivalent to approximate the kernel function
by using the Laplace transform where the discrete points on the contour line are the exponents.
A different idea for convolution quadratures can be the use of the SOE
approximation to the kernel function and the RK step for each exponential is then performed. The advantages
of this idea are twofold:

(i) An alternative and more efficient SOE method is introduced other than the Laplace transform for better approximation of some kernels. Here the ``better'' means that the maximal exponent of the exponentials can be tuned to an $O(1)$ constant instead of $O(1/h)$ in the Laplace transform-based methods with $h$ the step size of RK. The convergence order of the error bound with respect to $h$ is thus consistent with the order of RK method whereas the stage order of RK is not involved.

(ii) One can perform model reduction techniques to reduce the number of exponentials such that the computational cost can be saved.

\subsection{Fast convolution}\label{fastconvolution}
We then consider the evaluation of the convolution $y(t)$ given in Eq.\eqref{TimeConv}.
One first assumes that $f(\tau)$ has no singularity, then an SOE expansion of the kernel
by the VPMR reads,
\begin{equation}\label{41}
f(\tau)\approx f_{\text{es}}(\tau)=\sum_{\ell=1}^P m_\ell e^{-s_\ell\tau}, ~~\tau\in[0,t]
\end{equation}
with $m_\ell,s_\ell\in\mathbb{C}$ and the real part $\Re(s_{\ell})\geq 0$, which holds
$
\|f(\tau)-f_\text{es}(\tau)\|_{\infty}<\varepsilon
$
for a prescribed accuracy $0<\varepsilon\ll 1$.

The SOE expansion leads to an approximate representation of $y(t)$ by the summation of $P$ exponential integrals
\begin{equation}\label{exponential}
y(t)\approx\int_0^t f_{\text{es}}(t-\tau)g(\tau)d\tau=\sum_{\ell=1}^P m_{\ell}Y_\ell(t),
\end{equation}
with
\begin{equation}
Y_\ell(t)=\int_0^t e^{-s_\ell(t-\tau)}g(\tau)d\tau.
\end{equation}
Each term of
the above summation can be viewed as the solution at $\tau=t$ of the following ODE,
\begin{equation}\label{ODE1}
Y_{\ell}'(\tau)=-s_\ell Y_\ell(\tau)+g(\tau)~~\text{with}~Y(0)=0.
\end{equation}
Eq.\eqref{ODE1} can be efficiently computed via the RK method with step size $h$ (i.e., $t=Nh$)
to obtain a high-accurate solution within $O(N)$ operations.

To be specific, we employ implicit $S$-stage RK method of the form
\begin{equation}\label{RK}
Y_\ell^{n+1}=Y_\ell^n+h\sum_{i=1}^{S}b_iK_i,
\end{equation}
where
\begin{equation}\label{DefiK}
K_{i}=-s_{\ell}\left(Y_\ell^{n}+h\sum_{j=1}^{S}a_{ij}K_j\right)+g(t_{n}+c_{i}h),~~i=1,2,\cdots,S,
\end{equation}
with $a_{ij}$, $b_i$ and $c_i$ being coefficients, and $Y_\ell^n$ being the discretization of $Y_\ell(nh)$. We suppose that a RK method of order $p$ and stage order
$q$ is employed (follow the definitions of Ref. \cite{Lubich1993Runge}), namely, the local truncation error is $O(h^{p+1})$ and each internal stage error is $O(h^{q+1})$.
Let us denote $\mathcal{\bm{A}}=(a_{ij})_{S\times S}$, $\bm{\beta}^T=(b_1,\cdots, b_S)$ and $\bm{\zeta}=(c_1,\cdots,c_S)$.
They are usually arranged in a mnemonic device which is known as the Butcher tableau. The stability function of the RK method Eq.\eqref{RK} is defined as
\begin{equation}
r(z)=1+z\bm{\beta}^{T}(\bm{I}-z\mathcal{\bm{A}})^{-1}\bm{E}=\dfrac{\det(\bm{I}-z\mathcal{\bm{A}}+z\bm{E}\bm{\beta}^{T})}
{\det(\bm{I}-z\mathcal{\bm{A}})}.
\end{equation}
where $\bm{E}$ is the column vector of ones and $z=-s_{\ell}h$. We choose suitable implicit RK such that $b_{j}=a_{Sj}$ for $j=1,\cdots,S$,
$c_q=1$, and all eigenvalues of $\mathcal{\bm{A}}$ have positive real parts. This implies
that the method is L-stable, i.e.,
\begin{equation}\label{L-condition}
|r(z)|\leq 1~\text{for}~\Re(z)\leq 0~\text{and}~r(\infty)=0.
\end{equation}
This stability condition is important considering that ODE \eqref{ODE1} may become stiff
when the positive exponent is large.

Substituting Eq.\eqref{RK} into Eq.\eqref{ODE1}, one can express the  solution of the ODE as,
\begin{equation} \label{yl}
Y_\ell^{n+1}=h\sum_{j=0}^{n}\bm{v}_{n-j}(z_\ell)\bm{g}_j=r(z_{\ell})Y_\ell^{n}+h \bm{\psi}_\ell \bm{g}_{n},
\end{equation}
where $\bm{\psi}_\ell=\bm{\beta}^{T}(\bm{I}-z_{\ell}\mathcal{\bm{A}})^{-1}$, $z_\ell=-s_\ell h$, and $\bm{v}_{n}(z)$ and $\bm{g}_j$ are defined by,
\begin{equation}\label{precompute}
\bm{v}_n(z)=r(z)^{n}\bm{\beta}^T(\bm{I}-z\mathcal{\bm{A}})^{-1},~~~~
\bm{g}_j=(g(t_j+c_1h),\cdots,g(t_j+c_Sh))^{T}.
\end{equation}
Then,  the convolution integral Eq.\eqref{TimeConv} at time $t$ is approximated by,
\begin{equation}\label{final}
y(t)\approx \sum_{\ell=1}^Pm_{\ell}\left[r(z_{\ell})Y_\ell^{N-1}+h \bm{\psi}_\ell \bm{g}_{N-1}\right],
\end{equation}
with $Y_\ell^{N-1}$ being recursively solved using Eq.\eqref{yl}. Since $\bm{\psi}_\ell$ and $\bm{g}_{N-1}$
are row and column vectors of dimension $q$, the complexity of each time step is $O(P)$.

When the kernel $f(\tau)$ has a singularity (or near singular) at the origin, one shall remove
the singularity by splitting the integral into two parts. For a given $t_0 \ll 1$ and $T=t-t_0$,
the convolution integral is written as,
\begin{equation}\label{split}
y(t)=\int_{0}^{t_0}f(\tau)g(t-\tau)d\tau+\int_{0}^Tf(t-\tau)g(\tau)d\tau:=I_1+I_2.
\end{equation}
Note that $I_2$ has no singularity any more, thus it can be computed
using the aforementioned way. To evaluate $I_1$ over interval $[0,t_0]$, one first approximates $g(\tau)$
with polynomial interpolation, and $f(\tau)$ can be expanded by its generalized Taylor series
\begin{equation}\label{413}
f(\tau)= a_0\tau^{-\alpha}+a_1\tau+a_2\tau^2+...
\end{equation}
where $a_0\tau^{-\alpha}$ is the leading order asymptotic of the kernel, and may become weak singular with $0\leq\alpha<1$. For weak singular or nearly singular kernels, it is important to have some a priori asymptotic analysis around the singularity point. With such techniques, the contribution is simplified to a polynomial-polynomial convolution.
Take $t_0=O(h)$ and let $G(\tau)$ be the interpolation polynomial of $g(\tau)$. One has,
\begin{equation}\label{I_1}
I_{1}\approx\int_{0}^{t_0}\left[a_0\tau^{-\alpha}+a_1\tau+a_2\tau^2+\cdots\right] G(t-\tau)d\tau,
\end{equation}
which can be evaluated explicitly.

We summarize the method in Algorithm \ref{algorithm1}. Note that except steps 5, 6, 9 and 10, all of the other steps are precomputed. Moreover, the singular part $I_1$ does not depend on $t$ and is computed only once in the calculation. Thus, the overall computation cost for the convolution integral is $O(NP)$ where $N$ is the number of time steps.

\begin{remark}
	One important reason for traditional numerical methods having lower convergence order for weakly singular kernels is that the kernels can not be expanded as standard Taylor's series about the singularities. In order to approximate a kernel near its singularity, the generalized Taylor series Eq.\eqref{413} \cite{liao2003beyond,trujillo1999riemann,liu2015local,osler1971taylor} is introduced in this paper. We do not review the literature here, and one can refer to these papers for obtaining the detailed form of Eq.\eqref{413}. In some recent works, this technique is also introduced to generate fractional order degenerate kernel methods \cite{tongke2019fractional,guo2020fractional}, Taylor-collocation methods \cite{zarei2018solving}, and approaches for computing series solution \cite{toutounian2014new} for solving equation-based problems.
\end{remark}

\begin{algorithm}[h]
	\caption{Convolution quadrature based on the SOE}
	\label{algorithm1}
    \leftline{{\bf Input:}~Time $t$, quadrature kernel $f(\tau)$ and $g(\tau)$}
	\leftline{{\bf Output:}~Convolution quadrature $y(t)$ given in Eq.\eqref{TimeConv}}
	\begin{algorithmic}[1]
		\State Choose a suitable step size $h$ and the butcher table of the RK method
		\State Precompute the explicit expression of $\bm{v}_n(z)$ given in Eq.\eqref{precompute}
		\If{$f$ does not have singularity at $0$}
		\State Construct an SOE expansion of $f$ on $[0,t]$
		\State Compute $\{\bm{g}_j\}_{j=0}^N$
		\State Evaluate $y(t)$ according to Eq.\eqref{final}
		\Else
		\State Split $y$ into the sum of local integral $I_1$ and convolution $I_2$ for a specified $t_0$
		\State Explicitly approximate $I_1$ by Eq.\eqref{I_1}
		\State Employ steps 4-6 to evaluate $I_{2}$, then sum up $I_1$ and $I_2$
		\EndIf
	\end{algorithmic}
\end{algorithm}

\subsection{Error estimate}
We present the error analysis of the fast convolution algorithm.
We assume that the kernel $f(\tau)$ has singularity at the origin.
The solution $y_h(t)$ to approximate the decomposition in Eq.\eqref{split} is given by,
\begin{equation}
y_{h}(t)=I_1^{h}+I_2^{h},
\end{equation}
where $I_1^{h}$ and $I_{2}^{h}$ are numerical approximations of $I_{1}$ and $I_{2}$.

Consider the estimate of $I_{1}^{h}$. Suppose that $g(t-\tau)\in C^{\gamma}([0,t_0])$ with $\gamma$ being an integer
and $G(t-\tau)$ is a $L$-order interpolation approximation of $g(t-\tau)$ with $L\leq\gamma$.
By standard numerical analysis, the interpolation error is,
\begin{equation}
|G(t-\tau)-g(t-\tau)|\leq C_{0}\|g^{(L)}\|_{\infty}h^{L},~\forall \tau\in[0,t_0].
\end{equation}
If we truncate the generalized Taylor series Eq.\eqref{413} at $M$-th order such that $f(\tau)\approx f_{M}(\tau)$, the error estimate of the singular part $I_{1}$ reads,
\begin{equation}\label{estimateI1}
\begin{split}
|I_{1}-I_{1}^{h}|&=\left|\int_{0}^{t_0}\left(f(\tau)-f_M(\tau)\right)g(t-\tau)d\tau+\int_{0}^{t_0}f_M(\tau)\left(g(t-\tau)-G(t-\tau)\right)d\tau\right|\\
&\leq \int_{0}^{t_0}\left|f(\tau)-f_{M}(\tau)\right||g(t-\tau)|d\tau+\int_{0}^{t_0}\left|f_M(\tau)\right|\left|g(t-\tau)-G(t-\tau)\right|d\tau\\
&\leq C_{1}t_0^{M+1}\int_{0}^{t_0}|g(t-\tau)|d\tau+C_{0}\|g^{(L)}\|_{\infty}h^L\int_{0}^{t_0}\left|f_M(\tau)\right|d\tau\\
&\leq C_{1}(n_0h)^{M+1}\|g\|_{L^{1}}+C_{0}C_{f_M,t_0}\|g^{(L)}\|_{\infty}h^L
\end{split}
\end{equation}
where $C_{f_M,t_0}=\int_{0}^{t_0}|f_M(\tau)|d\tau$ is bounded because the convolution Eq.\eqref{TimeConv} is well defined, and $C_{0}$ and $C_{1}$ are constants.

For the exponential convolution part $I_2^{h}$, the kernel function $f(t-\tau)$ is approximated by its SOE expansion $f_{\text{es}}(t-\tau)$ on $[0,T]$
with the error tolerance $\varepsilon$, then the error reads,
\begin{equation}\label{estimateI2}
\begin{split}
\left|I_{2}-I_{2}^{h}\right|&=\left|\int_{0}^{T}(f(t-\tau)-f_{\text{es}}(t-\tau))g(\tau)d\tau+\sum_{\ell=1}^{P}m_{\ell}E^{\ell}_{\text{RK}}(t)\right|\\
&\leq \varepsilon\|g\|_{L^1}+Pm_{\max}\left|E^{\max}_{\text{RK}}(t)\right|
\end{split}
\end{equation}
where $E^{\ell}_{\text{RK}}$ is the error introduced by employing the RK method to solve the $\ell$-th ODE Eq.\eqref{ODE1}, $m_{\max}=\max\{|m_\ell|\}_{\ell=1}^{P}$ and $E_{\text{RK}}^{\max}(t)=\max\left\{ \left|E_{\text{RK}}^{\ell}(t)\right| \right\}_{\ell=1}^{P}$. To estimate $E_{\text{RK}}^{\max}(t)$, we firstly consider the following four lemmas. The proof of these lemmas can be found in \cite{banjai2011error,Lubich1993Runge}.
\begin{lemma}\label{lemma1}
When the RK method is applied to $y'=\lambda y+g$, the stability function $r(z)$ satisfies
\begin{equation}
    r(z)-e^z=O(z^{p+1})
\end{equation}
for $z\rightarrow 0$.
\end{lemma}

\begin{lemma}\label{lemma2}
Consider that the RK method is applied to $y'=\lambda y+g$ with $h$ the step size and $r(z)$ the stability function where $z=\lambda h$. We define
\begin{equation}
    f_n^{(k)}(z)=\sum_{\nu=0}^nr(z)^{-\nu}\nu^k,
\end{equation}
then the following upper bound
\begin{equation}
    |zr(z)^nf_n^{(k)}(z)|\leq C(n+|z|^{-1})^ke^{\Re(z) n}
\end{equation}
holds.
\end{lemma}

\begin{lemma}\label{lemma3}
For the RK method with order $p$, the following estimate
\begin{equation}
    \bm{\beta}^{T}(\bm{I}-z\mathcal{\bm{A}})^{-1}(k\mathcal{\bm{A}}\bm{\xi}^{k-1}-\bm{\xi}^k)=O(z^{p-k})~\text{as}~z\rightarrow 0
\end{equation}
holds for $k=1,2,\cdots,p$ where $\bm{\xi}^k=(c_1^k,c_2^k,\cdots,c_q^k)$ and the definitions of $\bm{\beta}$, $\bm{I}$, and $\mathcal{\bm{A}}$ follow Section \ref{fastconvolution}.
\end{lemma}

\begin{lemma}\label{lemma4}
The error at time $t_n=nh$ of the RK method with $h$ the step size applied to $y'=\lambda y+t^l/l!$, $y(0)=0$, $t\in[0,T]$, is given by
\begin{equation}\label{equation2}
    \begin{split}
    e_n&= \lambda^{-l-1}(r(h\lambda)^n-e^{nh\lambda})-\sum_{k=q+1}^ph^k\sum_{\nu=1}^{n-1}r_{n-1-\nu}^{(k)}(h\lambda)\lambda^{k-l-1}\sum_{\kappa=0}^{l-k}\dfrac{(\lambda \kappa h)^{\kappa}}{\kappa!}\\
    &=\lambda^{-l-1}(r(h\lambda)^n-e^{nh\lambda})-\tilde{e}_n
    \end{split}
\end{equation}
with $r(z)$ the stability function,
\begin{equation}
    r_n^{(k)}(z)=r(z)^nz\bm{\beta}^T(\bm{I}-z\mathcal{\bm{A}})^{-1}\delta^{(k)},
\end{equation}
and
\begin{equation}
    \delta^{(k)}=\mathcal{\bm{A}}\bm{\xi}^{k-1}-\bm{\xi}^k/k
\end{equation}
where $p$ and $q$ are the order and the stage order of the RK method, respectively. The definitions of $\bm{\beta}$, $\bm{I}$, $\mathcal{\bm{A}}$, and $\bm{\xi}^{k}$ follow Lemma \ref{lemma3}.
\end{lemma}

We have the following theorem to estimate $E_{\text{RK}}^{\max}(t)$.
\begin{theorem}\label{theorem1}
Suppose a $S$-stage implicit L-stable RK method. Let $p$ be the approximate order of the RK, and $q\leq p-1$ be the stage order satisfies condition Eq.\eqref{L-condition}. Let $h$ be the time step.
If $g(\tau)\in C^{(\gamma)}([0,t-t_0])$, $\gamma\geq p$ and $\max\limits_{\ell}|s_{\ell} h|\leq 1$, then the error $E_{\text{RK}}^{\max}(t)$ is bounded by
\begin{equation}\label{estimateI3}
\left|E_{\text{RK}}^{\max}(t)\right| \leq Ch^{p}\left( \sum_{\ell=0}^{p-1} \|g^{(\ell)}(0)\|_{\infty}+\max_{0\leq\tau\leq t-t_0}\|g^{(p)}(\tau)\|_{\infty} \right),
\end{equation}	
where $C$ is a constant.
\end{theorem}
\begin{proof}
For this proof, $C$ will denote a generic constant that is allowed to depend on $T$. We first consider the numerical error of which the RK method is applied to $y'=-s_{\ell} y+t^l/l!$, $y(0)=0$, where we recall $-s_{\ell}$ the exponent of the $\ell$-th exponential of the SOE constructed by the VPMR method. From Lemma \ref{lemma1}, for $nh\leq T$, we have
\begin{equation}\label{equation10}
	\left|r(-hs_{\ell})^n-e^{-nhs_{\ell}}\right|=\left|\left(r(-hs_{\ell})-e^{-hs_{\ell}}\right)\sum_{\nu=0}^{n-1}r(-hs_{\ell})^{n-1-\nu}e^{-\nu hs_{\ell}}\right|\leq Ch^{p},
\end{equation}
where the boundness of the stability function $|r(-hs_{\ell})^{n-1-\nu}|$ and $|e^{-\nu hs_{\ell}}|$ are employed. By Lemma \ref{lemma4}, with $\lambda=-s_{\ell}$, the error $e_n$ can be split into two parts as in Eq.\eqref{equation2}, where the first part has an estimate by applying Eq.\eqref{equation10} directly:
\begin{equation}\label{orderh}
    |(-s_{\ell})^{-l-1}\left(r(-hs_{\ell})^n-e^{-nhs_{\ell}}\right)|\leq|s_{\ell}|^{p-l-1}h^{p}=O(h^{p}).
\end{equation}
Eq.\eqref{orderh} takes the advantage of the VPMR method that $\max\limits_{\ell}{|s_{\ell}|}$ is restricted as $O(1)$, thus $|s_{\ell}|^{p-l-1}$ is $O(1)$.

Next, recalling the definition of $f_n^{(k)}(z)$ given in Lemma \ref{lemma2} and the expression of the second part $\tilde{e}_n$, we have
\begin{equation}\label{expressionen}
\begin{split}
    \tilde{e}_n&=\sum_{k=q+1}^ph^k\sum_{\nu=1}^{n-1}r_{n-1-\nu}^{(k)}(-s_{\ell}h)(-s_{\ell})^{k-l-1}\sum_{\kappa=0}^{l-k}\dfrac{(\lambda \kappa h)^{\kappa}}{\kappa!}\\
    &=h^{l+1}\sum_{k=q+1}^p\sum_{\nu=1}^{n-1}r(-s_{\ell}h)^{n-1-\nu}\bm{\beta}^T(\bm{I}+s_{\ell}h\mathcal{\bm{A}})^{-1}\delta^{(k)}(-s_{\ell}h)^{k-l}\sum_{\kappa=0}^{l-k}\dfrac{(-s_{\ell} \kappa h)^{\kappa}}{\kappa!}\\
    &=h^{l+1}\sum_{k=q+1}^p\bm{\beta}^T(\bm{I}+s_{\ell}h\mathcal{\bm{A}})^{-1}\delta^{(k)}(-s_{\ell}h)^{k-l-1}\sum_{\kappa=0}^{l-k}\dfrac{(-s_{\ell}h)^\kappa}{\kappa!}(-s_{\ell}h)r(-s_{\ell}h)^{n-1}f_n^{(\kappa)}(-s_{\ell}h).
\end{split}
\end{equation}
By substituting Lemmas \ref{lemma2} and \ref{lemma3} into Eq.\eqref{expressionen}, we obtain
\begin{equation}\label{inequality2}
	\begin{split}
	|\tilde{e}_n|&\leq Ch^{l+1}\sum_{k=q+1}^{p}|s_{\ell}h|^{p-k}\dfrac{|s_{\ell}h|^{k-l-1}}{k}\sum_{\kappa=0}^{l-k}\dfrac{|s_{\ell}h|^{\kappa}}{\kappa !}(n+|s_{\ell}h|^{-1})^{\kappa}e^{\Re(-s_{\ell}h)n}\\
	&\leq Ch^{p}|s_{\ell}|^{p-l-1}(1+n|s_{\ell}h|)^{l-q-1}\\
	&\leq Ch^{p}|s_{\ell}|^{p-q-2}=O(h^{p})
	\end{split}
\end{equation}
where the last inequality because $|s_{\ell}h|\leq 1$ and  $\max\limits_{\ell}|s_{\ell}|$ is $O(1)$.

From Eqs.\eqref{orderh} and \eqref{inequality2}, we obtain $e_n=O(h^p)$. The error bound for general smooth functions $g$ then follows with the Peano kernel argument of \cite{Lubich1993Runge}. Since $g(\tau)\in C^{(\gamma)}([0,T])$ with $\gamma\geq p$, we may treat each of the terms in the Taylor expansion of $g$ separately:
\begin{equation}\label{gtaylor}
	g(T)=\sum_{l=0}^{p-1}\dfrac{T^l}{l!}g^{(l)}(0)+\int_0^T\dfrac{\tau^{p-1}}{(p-1)!}g^{(p)}(T-\tau)d\tau.
\end{equation}
From \cite{Lubich1993Runge}, it is proved that the error of the RK method applied to the equation with inhomogeneity $g^{(l)}(0)t^{l}/l!$ is given by $e_{n}g^{(l)}(0)$.
By Eq.\eqref{gtaylor}, the error bound Eq.\eqref{estimateI3} could be obtained by estimating each term in the Taylor expansion of $g$ separately.

Combining Eqs.\eqref{estimateI1}, \eqref{estimateI2} and \eqref{estimateI3}, we find that the error satisfies $|y(t)-y_h(t)|=O\left(h^{d}+\varepsilon\right)$  for $d=\min\{M+1,~L,~p\}$.
\end{proof}

\begin{remark}
	Differently, the Laplace transform-based methods have a common error bound $O\left(h^{\min\{p,q+1\}}\right)$ \cite{Lubich1993Runge,banjai2011error,schadle2006fast} in which the stage order $q$ is also involved. This $O(h^{q+1})$ comes from the contour integral over the segment $|\Im (\tilde{\lambda})| \leq 1/h$ of the Laplace transform with $\lambda$ the integration path (see the proofs of Theorem 4.1 in \cite{banjai2011error} and Lemma 5.1 in \cite{Lubich1993Runge} for examples). Our convolution method takes advantage of the SOE constructed by the VPMR method, namely the maximal exponent which is independent of the step size $h$, can easily guarantee $|s_{\ell}h|\leq 1$. In fact, the proof of Theorem \ref{theorem1} is similar to the proof of Lemma 5.1 in \cite{Lubich1993Runge}, sharing the same $O(h^{p})$ convergence rate. This elimination of $O(h^{q+1})$ term will improve the convergence rate when $p>q+1$, as the Lobatto IIIC method (with $p=4$ and $q=2$) which is used in the numerical tests of this paper. In other words, our method avoids solving stiff problem via the RK.
\end{remark}

\begin{remark}\label{choosing}
	In Section \ref{fastconvolution}, we propose the $t_0=O(h)$ condition to guarantee  the $O(h^{M+1}+h^{L})$ decay of the error bound of approximation Eq.\eqref{I_1}. In practice, considering the error estimate in Eq.\eqref{estimateI1}, the choice of $t_0$ has a broader rule that $t_{0}^{M+1}\sim \varepsilon$ and $\int_{0}^{t_0}|f_{M}(\tau)|d\tau=O(1)$ which will not affect the overall accuracy and independent of $h$. In other words, choosing $t_0$ independent of $h$ has an additional advantage that one SOE can be used for different $h$. This rule for the choice of $t_0$ can be also used for the algorithms developed in Section \ref{convolutionquadrature}.
\end{remark}

\subsection{Numerical Examples}
We present numerical results to illustrate the performance of the SOE approximation method developed in this paper. We test the performance of the SOE to approximate the exact kernels $f(x)$ with the increase of $P$. To assess the accuracy, we compute the maximum error $E_{\infty}$ of the resulted SOE approximation $f_{\text{es}}(x)=\sum_{j=1}^{P}m_je^{-s_jx}$ with $P$. It can be viewed as an approximation of the continuous $L^{\infty}$ norm defined by,
\begin{equation}\label{mearsurement3}
	E_{\infty}=\max\left\{\left|f_{\text{es}}(x_{i})-f(x_{i})\right|,~i=1,\cdots,M\right\},
\end{equation}
where $\{x_{i},~i=1,\cdots,M\}$ are monitoring points randomly distributed from $(\delta,10]$ with $\delta\ll 1$ and we take $M=10000$. Second, the application and behavior of the fast convolution quadrature algorithm are illustrated from some examples in which the Lobatto IIIC \cite{Wanner1996} is employed as the RK method, i.e., with $S=3$, $p=4$ and $q=2$. The error and convergence order are used to measure the accuracy of the algorithm. The SOE for all different kernels are done with manually tuned parameters, which is optimized to obtain the required accuracy at a near-minimal number of exponentials. All the calculations in this section are performed with Matlab code on an Intel TM core of clock rate $2.50$ GHz with $24$ GB of memory.

In Eq.\eqref{exponential}, we take
a Gaussian kernel $f(\tau)=e^{-\tau^2/4}$,  $g(\tau)=\sin \tau$, and calculate the temporal convolution,
\begin{equation}\label{smoothconvolution}
y(t)=\int_{0}^{t}e^{-\frac{(t-\tau)^2}{4}}\sin \tau d\tau.
\end{equation}
The reference ``exact" solution is obtained via adaptive Gauss-Kronrod quadrature with $10^{-14}$ absolute accuracy.
Table \ref{Convtable} displays the results of $y$ at time $t=1, 4$ and $10$ for different
time steps. The SOE approximation parameters take $\varepsilon=8.1e-14$, $n_c/(2n-1)=1/8$ and the reduced number of exponentials is $P=20$.
A fourth order convergence in  $h$ is clearly shown, in agreement with the convergence order of the RK.
It is also found that the error does not accumulate with the increase of $t$. Figure \ref{Time} presents the CPU time as function of $t$,
and the linear scaling with respect to $t$ is illustrated.
\begin{table}[h]
	\centering
	\fontsize{10}{10}\selectfont
	\begin{threeparttable}
		\caption{Absolute errors and convergence rates to calculate convolution Eq.\eqref{smoothconvolution} for different $t$}
		\label{Convtable}
		\begin{tabular}{ccccccc}
			\toprule
			\midrule
Step size $h$ &$t=1$&Order&$t=4$&Order&$t=10$&Order\cr
			\midrule
			0.25& $4.49e-6$ & - & $3.31e-6$ & - & $3.53e-6$ & - \cr
			0.1& $1.19e-7$ & 3.93 & $1.03e-7$ & 3.62 & $1.06e-7$ & 3.70\cr
			0.05& $7.46e-9$ & 3.95 & $6.79e-9$ & 3.71 &$6.90e-9$ & 3.77\cr
			0.025&$4.68e-10$&3.96 & $4.36e-10$ & 3.77 &$4.41e-10$ & 3.82\cr
			0.01&$1.20e-11$&3.97 & $1.14e-11$ & 3.82 &$1.15e-11$ & 3.86\cr
			0.005&$7.21e-13$&3.98 & $6.96e-13$ & 3.85 &$7.10e-13$ &3.88\cr
			\midrule
			\bottomrule
		\end{tabular}
	\end{threeparttable}
\end{table}

\begin{figure}[h]
	\centering
	\includegraphics[width=0.6\textwidth]{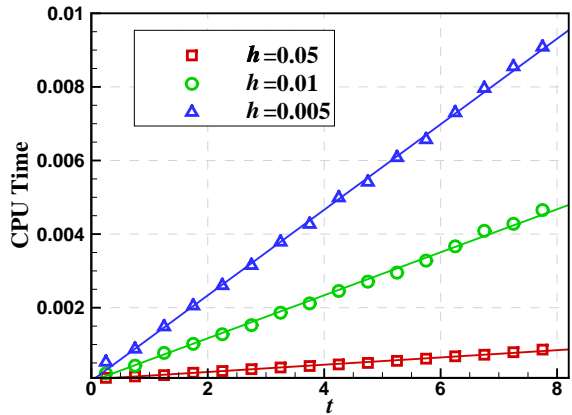}
	\caption{CPU time as function of $t$ for different time steps $h=0.05$ (red square), $0.01$ (green circle) and $0.005$ (blue triangle).
The solid lines are linear fitting of these data.}
	\label{Time}
\end{figure}

The next example studies the convolution of the singular power function $f(\tau)=\tau^{\alpha-1}$ with $0<\alpha<1$
and a smooth source $g(\tau)=\cos\tau$. The convolution is called the Riemann-Liouville fractional integral \cite{Srivastava2013} and is written as,
\begin{equation}\label{singularKernel}
y(t)=\dfrac{1}{\Gamma(\alpha)}\int_{0}^{t}(t-\tau)^{\alpha-1}\cos \tau d\tau,
\end{equation}
where $\Gamma(\alpha)$ is the Gamma function. This problem has exact solution,
\begin{equation}\label{exactsolution}
y(t)=\frac{2^{1-\alpha}\sqrt{\pi}t^{\alpha} }{\alpha\Gamma\left(\frac{\alpha}{2}\right)\Gamma\left(\frac{1+\alpha}{2}\right)}\widehat{\text{H}}\left(1,\left[\frac{1}{2}(1+\alpha),1+\frac{1}{2}\alpha\right],-\frac{t^{2}}{4}\right),
\end{equation}
where $\widehat{\text{H}}(\{a_{i}\}_{i=1}^{\eth_{1}},\{b_j\}_{j=1}^{\eth_{2}},\tau)$ is the generalized hypergeometric function defined as follows,
\begin{equation}
\widehat{\text{H}}(\{a_{i}\}_{i=1}^{\eth_{1}},\{b_j\}_{j=1}^{\eth_{2}},\tau)=\sum_{\ell=0}^{\infty}\left(\dfrac{\prod_{i=1}^{\eth_{1}}(a_{i})_{\ell}}{\prod_{j=1}^{\eth_{2}}(b_{j})_{\ell}}\right)\left(\dfrac{\tau^{\ell}}{\ell!}\right),
\end{equation}
with $\eth_{1}$ and $\eth_{2}$ being two positive integers and  $(\cdot)_{\ell}=\Gamma(\cdot+\ell)/\Gamma(\cdot)$ is the Pochhammer symbol.
In the SOE approximation, one takes $n_c$ such that $n_c/(2n-1)=0.15, 0.2$ and $0.4$ for $\alpha=0.1, 0.5$ and $0.9$, respectively. $t_0$ takes $0.05$, such that the values of $\int_{0}^{t_0}\tau^{\alpha-1}d\tau$ are $7.411$, $0.447$, and $0.075$ for $\alpha=0.1, 0.5$ and $0.9$, respectively.
And the number of exponentials keeps $P=640$. The error level of the approximation is $\sim 10^{-9}$.
In Eq.\eqref{split}, the convolution is split into $I_1$ and $I_2$.
We use a four-order scheme in $I_1$ and the Lobatto IIIC method for $I_2$.
The maximal absolute errors of $I$ for different $t$ are presented in Table \ref{Convtable2}, together with the corresponding convergence rates. Most of the data show approximate fourth-order convergence,
in agreement with the theoretical analysis, whereas some disagreements appear due to the possible accumulate of rounding error comes from the SOE or the interpolation, especially for $\alpha=0.1$ which corresponding to more singular nature.

It is noted that the power function as the kernel function requires a large $P$ to achieve high accuracy and the SOE approaches have been studied
in literature \cite{Beylkin2010Approximation,Yong2018The,Beylkin2005Approximation}. Typically, one shall use hundreds of exponentials to achieve
error of $8\sim9$ digits, but the maximal exponent is $\sim 10^3$ which dramatically decrease the convergence precision of RK.
The VPMR has slightly less number of exponentials for this level of accuracy, but with controllable upperbound of the positive exponents,
thus has better performance. We remark that the convolution quadrature will be improved if one introduces
better techniques such as the Ewald splitting (see \cite{JinLiXuZhao2020} for recent work and reference therein) and approximates the smooth part by the SOE expansion. This is not the central issue of this work,
and we save it for future study.

\begin{table}[h]
	\centering
	\fontsize{10}{10}\selectfont
	\begin{threeparttable}
		\caption{Absolute errors and convergence rates to calculate convolution Eq.\eqref{singularKernel} for different $t$}
		\label{Convtable2}
		\begin{tabular}{cccccccc}
			\toprule
			\midrule
            $\alpha$ &Step size $h$ &$t=1$&Order&$t=4$&Order&$t=8$&Order\cr
            \midrule
			\multirow{6}{*}{0.1}&0.25& $4.11e-5$ &-& $4.11e-5$ &-& $7.97e-5$ &-\cr
			&0.1& $4.61e-6$ & 3.64 & $1.73e-6$ & 3.45 & $3.04e-6$ & 3.57\cr
			&0.0625& $7.80e-7$ & 3.69 & $3.10e-7$ & 3.52 & $5.29e-7$ & 3.62\cr
			&0.05& $3.32e-7$ & 3.71 & $1.35e-7$ & 3.55 &$2.28e-7$ & 3.64\cr
			&0.025&$2.25e-8$&3.76 & $9.62e-9$ & 3.63 &$1.58e-8$ & 3.70\cr
			\midrule
			\multirow{6}{*}{0.5}&0.25& $4.22e-5$ &-& $1.02e-5$ &-& $2.41e-5$ &-\cr
			&0.1& $1.40e-6$ & 3.72 & $3.95e-7$ & 3.55 & $8.26e-7$ & 3.68\cr
			&0.0625& $2.31e-7$ & 3.76 & $6.85e-8$ & 3.61 & $1.39e-7$ & 3.72\cr
			&0.05& $9.75e-8$ & 3.77 & $2.94e-8$ & 3.63 &$5.92e-8$ & 3.73\cr
			&0.025&$6.55e-9$&3.81 & $2.40e-9$ & 3.63 &$4.34e-9$ & 3.74\cr
			\midrule
			\multirow{6}{*}{0.9}&0.25& $5.54e-6$ &-& $1.55e-6$ &-& $3.74e-6$ &-\cr
			&0.1& $1.69e-7$ & 3.81 & $4.82e-8$ & 3.78 & $1.09e-7$ & 3.86\cr
			&0.0625& $2.72e-8$ & 3.84 & $7.49e-9$ & 3.84 & $1.58e-8$ & 3.94\cr
			&0.05& $1.14e-8$ & 3.84 & $2.94e-9$ & 3.89 &$5.41e-9$ & 4.06\cr
			&0.025&$8.88e-10$&3.80 & $1.96e-10$ & 3.90 &$1.47e-9$ & 3.41\cr
			\midrule
			\bottomrule
		\end{tabular}
	\end{threeparttable}
\end{table}

\section{Convolution Quadrature Equations}\label{convolutionquadrature}
In this section, we extend the SOE-based fast convolution method to solve two kinds of convolution equations,
which has strong connection with the solution of many problems in applications \cite{Corduneanu1973,Srivastava2013}.
Our method provides both memory-saving and high-efficiency with $O(N)$ complexity for $N$ time steps, avoiding solving stiff problems.
\subsection{Linear convolution equations}\label{linearconvolution}
We consider  a class of linear convolution equations of the form,
\begin{equation}\label{Ht}
(1-\varpi) g(t)+H(t)=\int_0^tf(t-\tau)g(\tau)d\tau
\end{equation}
where $H(\tau)$ is a given source, $f(\tau)$ is the kernel function, $\varpi$ is a real parameter,
and $g(\tau)$ is unknown except the initial point $g(0)=g_0$. This class of integral equations has been
frequently studied, e.g., it can be derived from solving linear parabolic and hyperbolic evolution equations
by the method of lines \cite{lopez2013generalized}. When the kernel takes $f(t-\tau)=(t-\tau)^{-\alpha}$
with parameter $0<\alpha<1$, the equation Eq.\eqref{Ht} is the so-called Abel integral equation for $\varpi=1$
or the generalized Abel integral equation for $\varpi\neq1$. The equation Eq.\eqref{Ht} is the solution of the reduced Abel problem
and is essentially equivalent to the theory of generalized differentiation and integration \cite{Srivastava2013,Tamarkin1930}.

When the kernel function is not singular, we approximate $f(t-\tau)$ by its SOE expansion.
The discretization of Eq.\eqref{Ht} by the RK (as in Eq.\eqref{final}) at time $t=Nh$ reads,
\begin{equation}\label{Htd}
(1-\varpi) g_N + H_N=\sum_{\ell=1}^Pm_{\ell}\left[r(z_{\ell})Y_\ell^{N-1}+h \bm{\psi}_\ell \bm{g}_{N-1}\right]
\end{equation}
where $H_j=H(jh)$, $\bm{\psi}_\ell$ and $\bm{g}_{N-1}$ are defined in Section \ref{fastconv}. We define $g(t_j)$ as the integer stage and $g(t_j+ch)$ the internal stage when $0<c<1$. To reduce unknowns, one can approximate the internal stages by integer stages using the $m$-point interpolation,
\begin{equation}\label{interpol}
g(t_j+c_ih)=\sum_{\ell=1}^m\alpha^{i}_\ell g(t_{j-\ell+2})
\end{equation}
where $\alpha^i_\ell$ are interpolation weights. A recursive scheme is then formulated by solving a linear equation with one unknown quantity
at each time step, and thus the solution can be calculated explicitly.

If the kernel has a singularity at the origin, one employs the aforementioned splitting technique Eq.\eqref{split} to remove it.
Let $t_0=n_0h\ll1$ with integer $n_0$. When $t<t_0$, one replaces $f(t-\tau)$ by the generalized Taylor expansion  and $g(\tau)$ by
the local interpolation polynomial, respectively. The convolution equation then degenerates into a simple system of linear equations. When $t\geq t_0$, the convolution equation Eq.\eqref{Ht} reads,
\begin{equation}\label{t_t_0}
(1-\varpi)g(t) +H(t)=\int_0^{t_0}f(\tau)g(t-\tau)d\tau+\int_0^{T}f(t-\tau)g(\tau)d\tau.
\end{equation}
Approximating $f(t-\tau)$ by the SOE on $[0,T]$ and Eq.\eqref{t_t_0} can
be solved with exactly the same way as the nonsingular case.

We consider the approximation error for the case of singular kernels. For the singular part in the right hand of Eq.\eqref{t_t_0}, $f(\tau)$ is approximated
by $f_{M}(\tau)$ with $\tau^{-\alpha}$ the leading order asymptotic of the kernel, and $g(t-\tau)$ is approximated as an $L$-order interpolation. It is not difficult
to prove that the integration of singular integrand within $[0,t_0]$ has error of $O(h^{M+1}+h^{L})$.
For the smooth part in the right hand of Eq.\eqref{t_t_0}, kernel $f(t-\tau)$ is approximated by its $P$-term SOE expansion $f_{\text{es}}(t-\tau)$
within a prescribed $\varepsilon$ and the internal stages of $g(\tau)$ are approximated by an $m$-point
interpolation at the integer stages (as shown in Eq.\eqref{interpol}). With an $L$-stable $S$-stage RK method with order $p$ and $q+1\leq p$,
one has,
\begin{equation}\begin{split}
(1-\varpi) g(t)+H(t)=&\sum_{\ell=1}^{P}m_{\ell}e^{-s_\ell t_0}\int_{0}^{T}e^{-s_{\ell}(T-\tau)}g(\tau)d\tau+\int_{0}^{T}f_{\varepsilon}(t-\tau)g(\tau)d\tau\\
&+\int_0^{t_0}f(\tau)g(t-\tau)d\tau\\
=&\sum_{\ell=1}^{P}M_{\ell} Y_\ell(T)+\int_0^{t_0}f(\tau)g(t-\tau)d\tau+O(\varepsilon),
\end{split}\label{erroranalysis}
\end{equation}
where $M_\ell=m_\ell e^{-s_\ell t_0}$, $f_{\varepsilon}=f-f_{\text{es}}$, and $Y_\ell(T)$ is the solution of the ODE,
$
Y_{\ell}'(\tau)=-s_{\ell}Y_{\ell}(\tau)+g(\tau)~~\text{with}~Y_{\ell}(0)=0.
$
Suppose that $T=N_Th$. To clearly see the error behavior, by using Eq.\eqref{Htd} one rewrites Eq.\eqref{erroranalysis} as,	
\begin{equation}\label{erroranaysisI}
(1-\varpi) g(t)+H(t)=\sum_{\ell=1}^{P}M_{\ell}Y_\ell^{N_T}+\int_{0}^{t_0}f_{M}(\tau)G(t-\tau)d\tau+O(h^{d}+\varepsilon),
\end{equation}
where $Y_{\ell}^{N_T}$ is the numerical discretization of $Y_{\ell}(T)$ via RK. Note that $c_{q}=1$
and the stability function $r(z_\ell)$ satisfies $r(z_\ell)=e^{z_\ell}+O(h^{p+1})$ \cite{Wanner1996} thus it
is $O(1)$ for every $\ell$.

Let us write the middle term in Eq.\eqref{erroranaysisI} as $\displaystyle\int_{0}^{t_0}f_{M}(\tau)G(t-\tau)d\tau=R(t)+\kappa g(t),$
where the coefficient $\kappa$ of $g(t)$ is obviously $O(h^{1-\alpha})$. The solution of $g(t)$ is given by,
\begin{equation}
\begin{split}
g(t)&=\dfrac{1}{1-\varpi-\kappa}\left[h\sum_{\ell=1}^{P}\sum_{j=0}^{N_T-1}\sum_{s=1}^{q}M_{\ell}v_{N_T-1-j}^{s\ell}g_{j}^{s}+R(t)-H(t)+O(h^{d}+\varepsilon)\right]\\
&=\dfrac{1}{1-\varpi-\kappa}\left[h\sum_{j=0}^{N_T-1}\sum_{k=1}^{m}\xi_{jk}g(t_{j-k+2})+R(t)-H(t)+O(h^{d}+\varepsilon)\right],
\end{split}
\end{equation}
where $v_{N_T-1-j}^{s\ell}$ is the $s$-th component of $\bm{v}_{N_T-1-j}(z_\ell)$,  $g_{j}^{s}=g(t_{j}+c_{s}h)$, and
\begin{equation} \xi_{jk}=\sum_{\ell=1}^{P}\sum_{s=1}^{q}M_{\ell}v_{N_T-1-j}^{s\ell}\alpha_{k}^{s}
\end{equation}
is an $O(1)$ constant related
to the RK method and SOE coefficients. Since the internal point
interpolation is $O(h^{L})$, the error of $g(t)$ is bounded by
\begin{equation}
\gamma_{N}\leq\left|\frac{1}{1-\varpi-\kappa}\left[\sum\limits_{\ell=1}^{N_T}h\Xi_\ell(\gamma_{\ell}+O(h^{L}))+O(h^{d}+\varepsilon)\right]\right|,
\end{equation}
where $\gamma_{\ell}=\left|g(t_{\ell})-g_{\ell}\right|$ and $\Xi_\ell=\sum\limits_{j-k+2=\ell}\xi_{jk}$.

Obviously, if $\varpi=1$, the convergence rate for solving $g(t)$ is $O(h^{d-1+\alpha}+\varepsilon h^{\alpha-1})$. Otherwise, if $\varpi\neq 1$, the convergence rate is $O(h^{d}+\varepsilon)$.
Note that the error in the initial value will not affect
the convergence rate after a period of time.
\subsection{Nonlinear Volterra integral equations}
As another typical problem, we consider the nonlinear Volterra integral equation,
\begin{equation}\label{Volterra}
u(t)=a(t)+\int_0^tf(t-\tau)g(\tau,u(\tau))d\tau,~~t\geq 0
\end{equation}
where $g(\tau,u(\tau))$ is a smooth nonlinear function and $a(\tau)$ is an inhomogeneous known function. The unknown $u(\tau)$ is needed to be solved at the uniform time steps with the initial condition $u(0)=u_0$. The nonlinear Volterra integral equation Eq.\eqref{Volterra} arises in a variety of applications, including continuum mechanics, potential theory, electricity and magnetism \cite{jaswon1977integral,jiang2004second,levinson1960nonlinear}. The extension of the RK-based algorithm is
 less immediate, because the integral approximation uses the internal stages of the RK method.

We use the same setup as before with the kernel singularity at the origin, uniform time step $h$
and parameter $t_0=n_0h$ satisfying $0<t_0\ll1$. When $t<t_0$, similar as before, $f(t-\tau)$ and $g(t,u(t))$
are approximated by the generalized Taylor expansion and the local polynomial interpolation, respectively.
The numerical solution of $u(t)$ at the interpolating points on $[0,t_0]$ is then solved via Newton's iteration algorithm.
When $t\geq t_0$, $f(t-\tau)$ is approximated by its SOE expansion on $[0,T]$ such that
the solution is rewritten as the summation of a known function and the convolution of the exponentials and a nonlinear function.
We can introduce an RK-based convolution quadrature under the assumptions of Section \ref{fastconvolution},
then follow a recursive process to solve $u(t)$ step by step.
The discretization of the nonlinear Volterra equation
at time $t=Nh$ reads,
\begin{equation}
u(t)\approx u_N=a_N+\sum_{\ell=1}^PM_{\ell}\left[r(z_{\ell})Y_\ell^{N_T-1}+h \bm{\psi}_\ell \bm{g}_{N_T-1}\right]+\int_{0}^{t_0}f_{M}(\tau)G(t-\tau,u(t-\tau))d\tau.
\end{equation}
The values at internal stages $u(t_j+c_ih)$ with $0<c_i<1$ are approximated by interpolation at the integer points,
which reduces the cost of solving nonlinear equation.

We analyze the convergence rate of the algorithm.
The kernel $f(\tau)$ is approximated by its generalized $M-$order Taylor expansion $f_{M}(\tau)$ on $[0,t_0]$
and its SOE expansion $f_{\text{es}}(t-\tau)$ with tolerance $\varepsilon$ on $[0,T]$.
$g(t-\tau,u(t-\tau))$ is approximated by $L$-order interpolation $G(t-\tau,u(t-\tau))$ on $[0,t_0]$,
and the internal stages of $g(\tau)$ are approximated by its $m$-point interpolation at the integer points
on $[0,T]$. The numerical error is to estimate,
\begin{equation}
|u(t)-u_N|= \left|E_{t_0}+\int_{0}^{T}\left(f(t-\tau)-f_{\text{es}}(t-\tau)\right)g(\tau,u(\tau))d\tau + \sum_{\ell=1}^{P}E_{\text{RK}}^{\ell}(t)\right|
\end{equation}
where
\begin{equation}
E_{t_0}=\int_{0}^{t_0}\left[f(\tau)g(t-\tau,u(t-\tau))-f_{M}(\tau)G(t-\tau,u(t-\tau))\right]d\tau,
\end{equation}
and $E^{\ell}_{\text{RK}}(t)$ is the error introduced by employing
the RK method to solve the  ODE.  $E_{t_0}$ has an estimation,
\begin{equation}\label{ESTI1}
\begin{split}
|E_{t_0}|&=\left|\int_{0}^{t_0}(f(\tau)-f_{M}(\tau))g(t-\tau,u(t-\tau))+f_{M}(\tau)(g(t-\tau,u(t-\tau))-G(t-\tau,u(t-\tau)))d\tau\right|\\
&\leq C_{0}t_0^{M+1}\left|\int_{0}^{t_0}g(\tau,u(\tau))d\tau\right|+C_{1}h^{L}\sum_{i=0}^{L}\|\partial_t^i\partial_u^{L-i}g(\tau,u(\tau))\|_{\infty}\int_{0}^{t_0}\left|f_{M}(t-\tau)\right|d\tau\\
&\leq C_{0}(n_0h)^{M+1}\|g\|_{L_{1}}+C_{1}C_{f_{M},t_0}h^{L}\sum_{i=0}^{L}\|\partial_t^i\partial_u^{L-i}g(\tau,u(\tau))\|_{\infty}
\end{split}
\end{equation}
where $\displaystyle C_{f_M,t_0}=\int_{0}^{t_0}\left|f_M(t-\tau)\right|d\tau$ is bounded, and
$C_0$ and $C_1$ are constants. The error estimate of the exponential part reads,
\begin{equation}\label{ESTI2}
\left|\int_{0}^{T}(f(t-\tau)-f_{\text{es}}(t-\tau))g(\tau,u(\tau))d\tau\right|\leq \varepsilon \|g\|_{L_{1}}.
\end{equation}
For $E^{l}_{\text{RK}}(t)$, we have the following theorem.
\begin{theorem}\label{theorem4}
	For nonlinearity $g(\tau,u(\tau))$, assume that the following local Lipschitz condition for $\eta>0$,
	\begin{equation}\label{local}
	|g(\tau,v_1)-g(\tau,v_2)|\leq C(\eta)\cdot|v_1-v_2|~~\text{for}~|v_1|\leq \eta,~|v_2|\leq \eta,~0<\tau<t.
	\end{equation}
	Consider a $S$-stage implicit RK method with order $p$ and stage order $q\leq p-1$, satisfying stable condition Eq.\eqref{L-condition} and $\max\limits_{\ell}|s_{\ell} h|\leq 1$. If the internal stages are approximated by integer stages using the $m$-point $L$-order interpolation Eq.\eqref{interpol}, the error of RK is bounded by
	\begin{equation}\label{ESTI3}
	\left|E^{\ell}_{\text{RK}}(t)\right|\leq C_{2}(\eta)\left[\max_{j\in\{0,\cdots,N_T\}}\left|u_j-u(t_j)\right|+h^{p}+h^{L}\right].
	\end{equation}
\end{theorem}

\begin{proof}
	Similar to the analysis given in Section \ref{linearconvolution}, using the local Lipschitz condition, we obtain
\begin{equation}\label{identity}
	\begin{split}
	\left|E_{\text{RK}}^{\ell}(t)\right|&=\left|hM_{\ell}\sum_{j=0}^{N_T-1}\sum_{s=1}^{q}v_{N_T-1-j}^{s\ell}\left[\sum_{k=1}^{m}\alpha_{k}^{s}g(t_{j-k+2},u_{t_{j-k+2}})-g_{j}^{s}\right]+O(h^{p})\right|\\
	&=\left|hM_{\ell}\sum_{j=0}^{N_T-1}\sum_{s=1}^{q}v_{N_T-1-j}^{s\ell}\sum_{k=1}^{m}\alpha_{k}^{s}\left[g(t_{j-k+2},u_{t_{j-k+2}})-g_{j-k+2}^{0}\right]+O(h^{p}+h^{L})\right|\\
	&\leq \left|C(\eta)hM_{\ell}\sum_{j=0}^{N_T-1}\sum_{s=1}^{q}v_{N_T-1-j}^{s\ell}\sum_{k=1}^{m}\alpha_{k}^{s}\left|u_{j-k+2}-u(t_{j-k+2})\right|+O(h^{p}+h^{L})\right|\\
	&\leq C_{2}(\eta)\left[\max_{j\in\{0,\cdots,N_T\}}\left|u_j-u(t_j)\right|+h^{p}+h^{L}\right]
	\end{split}
\end{equation}
where $g_{j}^{s}=g(t_j+c_sh,u(t_j+c_sh))$ are the exact value of the stages. Note that the first identity in Eq.\eqref{identity} uses the condition $\max\limits_{\ell}|s_{\ell}h|\leq 1$ and $s_{\ell}$ being $O(1)$ for all $\ell$. Otherwise, the $O(h^{p})$ term should be replaced by $O(h^{\min\{p,q+1\}})$.

\end{proof}

The error bound of $\left|E^{\ell}_{\text{RK}}(t)\right|$ in Theorem \ref{theorem4} contains the maximum error of the solution at the time steps less than the current time. Choosing $t_0$ with the rule in Remark \ref{choosing} can naturally guarantee precision. Combining Eqs.\eqref{ESTI1}, \eqref{ESTI2} and \eqref{ESTI3}, we obtain the following error bound
\begin{equation}
|u(t)-u_n|= O(h^{d}+\varepsilon).
\end{equation}
\subsection{Examples}
In this subsection, there are four examples, among which the first two are linear integral equations and the third and the fourth are nonlinear integral equations. All the calculations in this section are performed with Matlab code on an Intel TM core of clock rate $2.50$ GHz with $24$ GB of memory.
\subsubsection{Linear convolution equation}
The first example is with the Gaussian kernel,
\begin{equation}\label{nonsin1.5}
	g(t)+\dfrac{\sqrt{\pi}}{2e}[f_1(t)+f_2(t)]-\cos(t)=\int_0^te^{-\frac{(t-\tau)^2}{4}}g(\tau)d\tau,
\end{equation}
with
\begin{equation}
	 f_1(t)=\left[\text{erf}\left(\dfrac{1}{2}(t-2i)\right)+\text{erf}\left(\dfrac{1}{2}(t+2i)\right)\right]\cos(t),
\end{equation}
\begin{equation}
	 f_2(t)=\left[-\text{erfi}\left(1-\dfrac{it}{2}\right)-\text{erfi}\left(1+\dfrac{it}{2}\right)+2\text{erfi}(1)\right]\sin(t).
\end{equation}
Here, $\text{erf}(\cdot)$ and $\text{erfi}(\cdot)$ are the error and imaginary error functions.
Eq.\eqref{nonsin1.5} has exact solution $g(t)=\cos{t}$. The parameters of RK and interpolation method are the same as the above case.
The SOE approximation parameters take $\varepsilon=8.1e-14$, $n_c/(2n-1)=1/8$ and $P=20$. A fourth order convergence in step size $h$ is shown, in agreement with the theoretical analysis.

\begin{table}[!h]
	\centering
	\fontsize{10}{10}\selectfont
	\begin{threeparttable}
		\caption{Absolute errors and convergence rates for the solution of Eq.\eqref{nonsin1.5} at different $t$}
		\label{Convtable7}
		\begin{tabular}{cccccccc}
			\toprule
			\midrule
			Step size $h$&$t=1$&Order&$t=4$&Order&$t=8$&Order\cr
			\midrule
			$0.1$&$3.25e-6$&-&$1.47e-5$&$-$&$1.71e-4$&$-$\cr
			$0.05$&$2.17e-7$&$3.91$&$9.50e-7$&$3.95$&$1.12e-5$&$3.94$\cr
			$0.025$&$1.41e-8$&$3.92$&$6.16e-8$&$3.95$&$7.27e-7$&$3.94$\cr
			$0.01$&$3.73e-10$&$3.94$&$1.62e-9$&$3.96$&$1.92e-8$&$3.95$\cr
			$0.005$&$2.35e-11$&$3.95$&$1.02e-10$&$3.96$&$1.21e-9$&$3.96$\cr
			$0.0025$&$1.71e-12$&$3.92$&$6.86e-12$&$3.95$&$8.27e-11$&$3.94$\cr
			\midrule
			\bottomrule
		\end{tabular}
	\end{threeparttable}
\end{table}

The second example is the generalized Abel equation, which has a singular kernel,
\begin{equation}\label{sin1.5}
	3g(t)+H(t)=\int_{0}^{t}(t-\tau)^{-\alpha}g(\tau)d\tau,
\end{equation}
where $H(t)$ is obtained with respect to higher accuracy via our algorithm by taking $g(\tau)=\cos{\tau}$. For the purpose of numerical test, $g(\tau)$ is solved
by the our numerical scheme, which is compared to $\cos \tau$ for the measurement of the error. The SOE approximation parameters take $\varepsilon=1.4e-8$, $n_c/(2n-1)=0.2$ and $P=600$. The $t_0$ is fixed as $0.05$.
Table \ref{Convtable8} displays the error and convergence rate for $\alpha=0.5$, and again one can observe the errors in agreement with the analysis.
\begin{table}[h]
	\centering
	\fontsize{10}{10}\selectfont
	\begin{threeparttable}
		\caption{Absolute errors and convergences rates for the solution of Eq.\eqref{sin1.5} at different $t$}
		\label{Convtable8}
		\begin{tabular}{ccccccc}
			\toprule
			\midrule
			Step size $h$&$t=2$&Order&$t=6$&Order&$t=10$&Order\cr
			\midrule
			$0.025$&$2.60e-8$&-&$9.80e-6$&$-$&$3.95e-7$&$-$\cr
			$0.01$&$1.13e-8$&$3.74$&$4.25e-8$&$3.74$&$1.71e-7$&$3.74$\cr
			$0.00625$&$1.47e-9$&$4.14$&$5.20e-9$&$4.24$&$1.97e-8$&$4.33$\cr
			$0.005$&$5.30e-10$&$4.25$&$1.90e-9$&$4.30$&$6.80e-9$&$4.33$\cr
			\midrule
			\bottomrule
		\end{tabular}
	\end{threeparttable}
\end{table}

\subsubsection{Nonlinear Volterra integral equation}
The third example arises in the analysis of neural networks with post-inhibitory rebound, where the model \cite{Heiden2013}  is given by,
\begin{equation}\label{SmoothVolterra}
u(t)=1+\int_{0}^{t}(t-\tau)^{3}(4-t+\tau)e^{-t+\tau}\dfrac{u^{4}(\tau)}{1+2u^{2}(\tau)+2u^{4}(\tau)}d\tau.
\end{equation}
The ``exact" solution of Eq.\eqref{SmoothVolterra} at $t=10$ is $u(10)=1.25995582337$ \cite{Heiden2013}.
We calculate the absolute error at time $t=10$ by our algorithm. The interpolation order takes $L=4$.
The SOE parameters are $\varepsilon=10^{-12}$, $n_c/(2n-1)=2.25$ and $P=170$.
The error tolerance for the Newton's method is $10^{-12}$. The accuracy results and CPU time performance (in seconds) are shown in Table \ref{Convtable5}
for different steps, from which the fourth-order convergence of the algorithm is displayed, in agreement with the
theoretical analysis. Due to the number of interative steps of Newton's method is different in each time step, the computational time achieves nearly linear scaling with $h$, demonstrating an attractive performance.

\begin{table}[h]
	\centering
	\fontsize{10}{10}\selectfont
	\begin{threeparttable}
		\caption{Absolute errors, convergence rates and CPU time (seconds) for solving the Volterra equation Eq.\eqref{SmoothVolterra}}
		\label{Convtable5}
		\begin{tabular}{cccc}
			\toprule
			\midrule
			Step size $h$ & ~~~Error~~~ &~~~Order~~~& CPU time \cr
			\midrule
			$1$&$2.65e-2$&$-$&$1.2e-4$\cr
			$0.625$&$3.91e-3$&$3.88$&$1.8e-4$\cr
			$0.5$&$1.44e-3$&$4.02$&$2.7e-4$\cr
			$0.25$&$4.64e-5$&$4.43$&$4.7e-4$\cr
			$0.0625$&$2.48e-7$&$4.12$&$1.3e-3$\cr
			$0.05$&$1.43e-7$&$4.01$&$1.8e-3$\cr
			$0.01$&$1.90e-10$&$4.05$&$7.4e-3$\cr
			\midrule
			\bottomrule
		\end{tabular}
	\end{threeparttable}
\end{table}

The fourth one is with a singular kernel, arising in the theory of superfluidity \cite{levinson1960nonlinear}. The equation
is given by,
\begin{equation}\label{SingularVolterra}
u(t)=-\int_{0}^{t}\dfrac{(u(\tau)-\sin\tau)^{3}}{\sqrt{\pi(t-\tau)}}d\tau.
\end{equation}
In the calculations, we take $t_0=0.05$, $\varepsilon=1.40e-8$, $n_c/(2n-1)=0.4$ and $P=640$. The order of interpolation is $L=4$.
The ``exact" solution is calculated with respect to $h=0.0001$ by our algorithm, with a tolerance of $10^{-10}$ in the iteration method.
Table \ref{Convtable6} displays the errors at $t=2,6$ and $10$ and the corresponding convergence rates. As expected,
a fourth order of convergence is observed, in agreement with theoretical analysis. A few disagreements appear due to the accuracy of the SOE is $10^{-10}$ which is dominated in the error when the step size is small.

\begin{table}[h]
	\centering
	\fontsize{10}{10}\selectfont
	\begin{threeparttable}
		\caption{Absolute errors and convergence rates of the Volterra integral equation Eq.\eqref{SingularVolterra} at different $t$}
		\label{Convtable6}
		\begin{tabular}{ccccccc}
			\toprule
			\midrule
			Step size $h$&$t=2$&Order&$t=6$&Order&$t=10$&Order\cr
			\midrule
			$0.025$&$3.33e-8$&-&$1.31e-7$&$-$&$7.39e-8$&$-$\cr
			$0.0125$&$2.00e-9$&$4.06$&$8.39e-9$&$3.97$&$4.14e-9$&$4.16$\cr
			$0.01$&$8.78e-10$&$3.97$&$3.63e-9$&$3.92$&$1.74e-9$&$4.09$\cr
			$0.00625$&$1.84e-10$&$3.82$&$6.75e-10$&$3.80$&$1.72e-10$&$4.37$\cr
			\midrule
			\bottomrule
		\end{tabular}
	\end{threeparttable}
\end{table}

\section{Conclusions}
We propose an accurate SOE approximation method VPMR for general kernels and develop
an accurate and fast algorithm for the temporal convolution, and integral equations with
convolution kernels. The SOE is constructed by a combination of the VP sum and the
MR method. The VPMR is accurate and efficient with controllable maximal exponents.
As applications of the VPMR, the convolution evaluation is computed with $O(N)$ operations
on uniform $N$ time steps owing to the kernel approximation with exponentials enabling a
recurrence formula solved by L-stable RK methods. The kernel singularity can be treated by
the splitting of the convolution such that the singular part can be calculated by analytical
techniques. The controllable of maximal exponents in the VPMR ensure that the efficiency of
proposed algorithm. Our algorithm is friendly for parallelization and favors easy extensions
to complicated kernel with the SOE. Numerical results for different kernels show its efficiency, accuracy and universality, demonstrating the attractive features for potential applications
in many problems.

\section*{Acknowledgment}
The authors acknowledge the financial support from the National Natural Science Foundation of China (grant Nos. 12071288 and 21773165),
Science and Technology Commission of Shanghai Municipality (grant Nos. 20JC1414100 and 21JC1403700), Strategic Priority Research Program of Chinese Academy of Sciences (grant No. XDA25010403),
and the support from the HPC center of Shanghai Jiao Tong University.

\section*{Declarations}

\subsection*{Conflict of interest}

The authors declare that they have no conflict of interest.

\subsection*{Data Availibility Statement}

The data that support the findings of this study are available from the corresponding author upon reasonable request.


\begin{thebibliography}{10}
\expandafter\ifx\csname url\endcsname\relax
  \def\url#1{\texttt{#1}}\fi
\expandafter\ifx\csname urlprefix\endcsname\relax\def\urlprefix{URL }\fi
\expandafter\ifx\csname href\endcsname\relax
  \def\href#1#2{#2} \def\path#1{#1}\fi

\bibitem{Beylkin2009Fast}
G.~Beylkin, C.~Kurcz, L.~Monz{\'o}n, Fast convolution with the free space
  {Helmholtz Green's} function, Journal of Computational Physics 228~(8) (2009)
  2770--2791.

\bibitem{H1999A}
H.~Cheng, L.~Greengard, V.~Rokhlin, A fast adaptive multipole algorithm in
  three dimensions, Journal of Computational Physics 155~(2) (1999) 468--498.

\bibitem{Yong2018The}
L.~Greengard, S.~Jiang, Y.~Zhang, The anisotropic truncated kernel method for
  convolution with free-space {Green's} functions, SIAM Journal on Scientific
  Computing 40~(6) (2018) A3733--A3754.

\bibitem{Jiang2008Efficient}
S.~Jiang, L.~Greengard, Efficient representation of nonreflecting boundary
  conditions for the time-dependent {Schr{\"o}dinger} equation in two
  dimensions, Communications on Pure and Applied Mathematics 61~(2) (2008)
  261--288.

\bibitem{Lubich2002}
C.~Lubich, A.~Sch{\"a}dle, Fast convolution for nonreflecting boundary
  conditions, SIAM Journal on Scientific Computing 24~(1) (2002) 161--182.

\bibitem{wang2020taylor}
B.~Wang, D.~Chen, B.~Zhang, W.~Zhang, M.~H. Cho, W.~Cai, {Taylor expansion
  based fast multipole method for 3-D Helmholtz equations in layered media},
  Journal of Computational Physics 401 (2020) 109008.

\bibitem{fang1988discrete}
D.~Fang, J.~Yang, G.~Delisle, Discrete image theory for horizontal electric
  dipoles in a multilayered medium, in: IEE Proceedings H-microwaves, Antennas
  and Propagation, Vol. 135, IET, 1988, pp. 297--303.

\bibitem{alparslan2010closed}
A.~Alparslan, M.~I. Aksun, K.~A. Michalski, {Closed-form Green's functions in
  planar layered media for all ranges and materials}, IEEE Transactions on
  Microwave Theory and Techniques 58~(3) (2010) 602--613.

\bibitem{spivak2010fast}
M.~Spivak, S.~K. Veerapaneni, L.~Greengard, {The fast generalized Gauss
  transform}, SIAM Journal on Scientific Computing 32~(5) (2010) 3092--3107.

\bibitem{2020onekernel}
Y.~Zhang, C.~Zhuang, S.~Jiang, Fast one-dimensional convolution with general
  kernels using sum-of-exponential approximation, Communications in
  Computational Physics 29~(5) (2021) 1570--1582.

\bibitem{Beylkin2005Approximation}
G.~Beylkin, L.~Monz{\'o}n, On approximation of functions by exponential sums,
  Applied and Computational Harmonic Analysis 19~(1) (2005) 17 -- 48.

\bibitem{Beylkin2010Approximation}
G.~Beylkin, L.~Monz{\'o}n, Approximation by exponential sums revisited, Applied
  and Computational Harmonic Analysis 28~(2) (2010) 131--149.

\bibitem{Braess1995Asymptotics}
D.~Braess, Asymptotics for the approximation of wave functions by exponential
  sums, Journal of Approximation Theory 83~(1) (1995) 93--103.

\bibitem{Dietrich2005Approximation}
D.~Braess, W.~Hackbusch, Approximation of $1/x$ by exponential sums in
  $[1,\infty)$, IMA Journal of Numerical Analysis 25~(4) (2005) 685--697.

\bibitem{Braess2009On}
D.~Braess, W.~Hackbusch, On the efficient computation of high-dimensional
  integrals and the approximation by exponential sums, in: Multiscale,
  nonlinear and adaptive approximation, Springer, 2009, pp. 39--74.

\bibitem{Evans1980On}
J.~W. Evans, W.~B. Gragg, R.~J. LeVeque, On least squares exponential sum
  approximation with positive coefficients, Mathematics of Computation 34~(149)
  (1980) 203--211.

\bibitem{articleGonchar}
A.~A. Gonchar, E.~A. Rakhmanov, Equilibrium distributions and degree of
  rational approximation of analytic functions, Mathematics of the USSR-Sbornik
  62~(2) (1989) 305.

\bibitem{jiang2019fast}
S.~Jiang, A fast {G}auss transform in one dimension using sum-of-exponentials
  approximations, arXiv: 1909.09825 (2019).

\bibitem{hamming2012numerical}
R.~Hamming, {Numerical Methods for Scientists and Engineers}, Courier
  Corporation, 2012.

\bibitem{Wiscombe1977Exponential}
W.~J. Wiscombe, J.~W. Evans, Exponential-sum fitting of radiative transmission
  functions, Journal of Computational Physics 24~(4) (1977) 416 -- 444.

\bibitem{boyd2010uselessness}
The uselessness of the {Fast Gauss Transform} for summing {Gaussian} radial
  basis function series, Journal of Computational Physics 229~(4) (2010) 1311
  -- 1326.

\bibitem{Lubich1993Runge}
C.~Lubich, A.~Ostermann, {Runge-Kutta methods for parabolic equations and
  convolution quadrature}, Mathematics of Computation 60~(201) (1993) 105--131.

\bibitem{Al2001Approximation}
R.~Albtoush, K.~Al-Khaled, Approximation of periodic functions by {Vall{\'e}e
  Poussin sums}, Hokkaido Mathematical Journal 30~(2) (2001) 269--282.

\bibitem{de1919leccons}
C.~J. de~La~Vall{\'e}e-Poussin, Le{\c{c}}ons sur l'approximation des fonctions
  d'une variable r{\'e}elle, Paris, 1919.

\bibitem{Moore1981Principal}
B.~Moore, Principal component analysis in linear systems: controllability,
  observability, and model reduction, IEEE Transactions on Automatic Control
  26~(1) (1981) 17--32.

\bibitem{Peter2015SIAM}
P.~Benner, S.~Gugercin, K.~Willcox, A survey of projection-based model
  reduction methods for parametric dynamical systems, SIAM Review 57~(4) (2015)
  483--531.

\bibitem{schadle2006fast}
A.~Sch{\"a}dle, M.~L{\'o}pez-Fern{\'a}ndez, C.~Lubich, Fast and oblivious
  convolution quadrature, SIAM Journal on Scientific Computing 28~(2) (2006)
  421--438.

\bibitem{banjai2011error}
L.~Banjai, C.~Lubich, {An error analysis of Runge--Kutta convolution
  quadrature}, BIT Numerical Mathematics 51~(3) (2011) 483--496.

\bibitem{2020kernel}
J.~Liang, Z.~Gao, Z.~Xu, {A kernel-independent sum-of-Gaussians method by de la
  Vallee-Poussin sums}, Advances in Applied Mathematics and Mechanics 13~(5)
  (2021) 1126--1141.

\bibitem{Antoulas2001}
A.~Antoulas, D.~Sorensen, Approximation of large-scale dynamical systems: an
  overview, International Journal of Applied Mathematics and Computer Science
  11~(5) (2001) 1093--1121.

\bibitem{GLOVER1984}
K.~Glover, All optimal hankel-norm approximations of linear multivariable
  systems and their ${L}^\infty$-error bounds, International Journal of Control
  39~(6) (1984) 1115--1193.

\bibitem{Raimund1991}
R.~Ober, Balanced parametrization of classes of linear systems, SIAM Journal on
  Control and Optimization 29~(6) (1991) 1251--1287.

\bibitem{MP}
B.~Barrowes, {Multiple Precision Toolbox for MATLAB}, MATLAB Central File
  Exchange (Retrieved August 10, 2020).

\bibitem{liu1998model}
W.~Liu, V.~Sreeram, K.~L. Teo, Model reduction for state-space symmetric
  systems, Systems \& Control Letters 34~(4) (1998) 209--215.

\bibitem{shampine2008vectorized}
L.~F. Shampine, {Vectorized adaptive quadrature in MATLAB}, Journal of
  Computational and Applied Mathematics 211~(2) (2008) 131--140.

\bibitem{occorsio2018cubature}
D.~Occorsio, G.~Serafini, Cubature formulae for nearly singular and highly
  oscillating integrals, Calcolo 55~(1) (2018) 1--33.

\bibitem{krishnamoorthy2013matrix}
A.~Krishnamoorthy, D.~Menon, {Matrix inversion using Cholesky decomposition},
  in: 2013 signal processing: Algorithms, architectures, arrangements, and
  applications (SPA), IEEE, 2013, pp. 70--72.

\bibitem{halko2011finding}
N.~Halko, P.-G. Martinsson, J.~A. Tropp, {Finding structure with randomness:
  Probabilistic algorithms for constructing approximate matrix decompositions},
  SIAM Review 53~(2) (2011) 217--288.

\bibitem{rokhlin2010randomized}
V.~Rokhlin, A.~Szlam, M.~Tygert, A randomized algorithm for principal component
  analysis, SIAM Journal on Matrix Analysis and Applications 31~(3) (2010)
  1100--1124.

\bibitem{lebesgue1909integrales}
H.~Lebesgue, {Sur les int{\'e}grales singuli{\`e}res}, in: {Annales de la
  Facult{\'e} des sciences de Toulouse: Math{\'e}matiques}, Vol.~1, 1909, pp.
  25--117.

\bibitem{zakharov1968bound}
A.~Zakharov, {Bound on deviations of continuous periodic functions from their
  De La Vall{\'e}e-Poussin sums}, Mathematical Notes of the Academy of Sciences
  of the USSR 3~(1) (1968) 45--49.

\bibitem{stechkin1961approximation}
S.~B. Stechkin, {The approximation of periodic functions by Fej{\'e}r sums},
  Trudy Matematicheskogo Instituta imeni VA Steklova 62 (1961) 48--60.

\bibitem{efimov1959approximation}
A.~V. Efimov, {Approximation of periodic functions by de la Vallee-Poussin
  sums}, Izvestiya Rossiiskoi Akademii Nauk Seriya Matematicheskaya 23~(5)
  (1959) 737--770.

\bibitem{telyakovskii1958approximation}
S.~A. Telyakovskii, {Approximation of differentiable functions by de la
  Vall{\'e}e Poussin's sums}, in: Doklady Akademii Nauk, Vol. 121, Russian
  Academy of Sciences, 1958, pp. 426--429.

\bibitem{nikolski1940certaines}
S.~Nikolski, {Sur certaines m{\'e}thodes d'approximation au moyen de sommes
  trigonom{\'e}triques}, Izvestiya Rossiiskoi Akademii Nauk. Seriya
  Matematicheskaya 4~(6) (1940) 509--520.

\bibitem{boyer2011generalized}
R.~P. Boyer, W.~M.~Y. Goh, Generalized {Gibbs} phenomenon for {Fourier} partial
  sums and de la {Vall{\'e}e-Poussin} sums, Journal of Applied Mathematics and
  Computing 37~(1-2) (2011) 421--442.

\bibitem{magomed2016approximation}
M.~G. Magomed-Kasumov, {Approximation properties of de la Vall{\'e}e-Poussin
  means for piecewise smooth functions}, Mathematical Notes 100~(1) (2016)
  229--244.

\bibitem{sharapudinov2012approximation}
I.~I. Sharapudinov, {Approximation properties of de la Vall{\'e}e-Poussin means
  on classes of Sobolev type with variable exponent}, Vestn. Daghestan Res.
  Center Russian Academy of Sciences 45 (2012) 5--13.

\bibitem{huang2006chain}
H.~Huang, S.~Marcantognini, N.~Young, Chain rules for higher derivatives, The
  Mathematical Intelligencer 28~(2) (2006) 61--69.

\bibitem{SpivakThe}
M.~Spivak, S.~K. Veerapaneni, L.~Greengard, The fast generalized {Gauss}
  transform, SIAM Journal on Scientific Computing 32~(5) (2010) 3092--3107.

\bibitem{lopez2006spectral}
M.~L{\'o}pez-Fern{\'a}ndez, C.~Palencia, A.~Sch{\"a}dle, A spectral order
  method for inverting sectorial {L}aplace transforms, SIAM Journal on
  Numerical Analysis 44~(3) (2006) 1332--1350.

\bibitem{trefethen2006talbot}
L.~N. Trefethen, J.~A.~C. Weideman, T.~Schmelzer, Talbot quadratures and
  rational approximations, {BIT Numerical Mathematics} 46~(3) (2006) 653--670.

\bibitem{talbot1979accurate}
A.~Talbot, The accurate numerical inversion of {L}aplace transforms, {IMA
  Journal of Applied Mathematics} 23~(1) (1979) 97--120.

\bibitem{weideman2007parabolic}
J.~Weideman, L.~Trefethen, Parabolic and hyperbolic contours for computing the
  bromwich integral, Mathematics of Computation 76~(259) (2007) 1341--1356.

\bibitem{weideman2010improved}
J.~Weideman, Improved contour integral methods for parabolic {PDEs}, IMA
  Journal of Numerical Analysis 30~(1) (2010) 334--350.

\bibitem{PronyToolbox}
S.~Singh, {Prony Toolbox}, MATLAB Central File Exchange (Retrieved September 6,
  2021).

\bibitem{RogerInterpolation}
R.~Woodard, Interpolation of spatial data: {S}ome theory for kriging, Springer,
  1999.

\bibitem{Williams2005Gaussian}
C.~K.~I. Williams, {G}aussian {P}rocesses for {M}achine {L}earning, MIT Press,
  2005.

\bibitem{Alexander2018Gaussian}
A.~Denzel, J.~K{\"a}stner, Gaussian process regression for geometry
  optimization, The Journal of Chemical Physics 148~(9) (2018) 094114.

\bibitem{Dral2019MLatom}
P.~O. Dral, Gaussian process regression for geometry optimization, Journal on
  Computational Chemistry 40~(26) (2019) 2339--2347.

\bibitem{ewald1921berechnung}
P.~P. Ewald, Die {B}erechnung optischer und elektrostatischer
  {G}itterpotentiale, Ann. Phys. 369~(3) (1921) 253--287.

\bibitem{JinLiXuZhao2020}
S.~Jin, L.~Li, Z.~Xu, Y.~Zhao, {A random batch Ewald method for particle
  systems with Coulomb interactions}, SIAM Journal on Scientific Computing
  43~(4) (2021) B937--B960.

\bibitem{lopez1990difference}
M.~Lopez-Marcos, A difference scheme for a nonlinear partial
  integro-differential equation, SIAM Journal on Numerical Analysis 27~(1)
  (1990) 20--31.

\bibitem{sanz1988numerical}
J.~M. Sanz-Serna, A numerical method for a partial integro-differential
  equation, SIAM Journal on Numerical Analysis 25~(2) (1988) 319--327.

\bibitem{CAO2013154}
J.~Cao, C.~Xu, A high order schema for the numerical solution of the fractional
  ordinary differential equations, Journal of Computational Physics 238 (2013)
  154 -- 168.

\bibitem{cuesta2003fractional}
E.~Cuesta, C.~Palencia, {A fractional trapezoidal rule for integro-differential
  equations of fractional order in Banach spaces}, Applied Numerical
  Mathematics 45~(2-3) (2003) 139--159.

\bibitem{diethelm2002analysis}
K.~Diethelm, N.~J. Ford, Analysis of fractional differential equations, Journal
  of Mathematical Analysis and Applications 265~(2) (2002) 229--248.

\bibitem{banjai2019efficient}
L.~Banjai, M.~L{\'o}pez-Fern{\'a}ndez, Efficient high order algorithms for
  fractional integrals and fractional differential equations, Numerische
  Mathematik 141~(2) (2019) 289--317.

\bibitem{li2010fast}
J.-R. Li, A fast time stepping method for evaluating fractional integrals, SIAM
  Journal on Scientific Computing 31~(6) (2010) 4696--4714.

\bibitem{ZAKERI20106548}
G.-A. Zakeri, M.~Navab, {Sinc collocation approximation of non-smooth solution
  of a nonlinear weakly singular Volterra integral equation}, Journal of
  Computational Physics 229~(18) (2010) 6548 -- 6557.

\bibitem{lubich1985fractional}
C.~Lubich, {Fractional linear multistep methods for Abel-Volterra integral
  equations of the second kind}, Mathematics of Computation 45~(172) (1985)
  463--469.

\bibitem{miller1975volterra}
R.~K. Miller, {Volterra integral equations in a Banach space}, Funkcial. Ekvac
  18~(2) (1975) 163--193.

\bibitem{MOHAMMADI2015254}
F.~Mohammadi, A wavelet-based computational method for solving stochastic
  {It\^{o} -- Volterra} integral equations, Journal of Computational Physics
  298 (2015) 254 -- 265.

\bibitem{lubich1988convolution}
C.~Lubich, {Convolution quadrature and discretized operational calculus. I},
  Numerische Mathematik 52~(2) (1988) 129--145.

\bibitem{lubich1988convolution2}
C.~Lubich, {Convolution quadrature and discretized operational calculus. II},
  Numerische Mathematik 52~(4) (1988) 413--425.

\bibitem{lubich2004convolution}
C.~Lubich, Convolution quadrature revisited, BIT Numerical Mathematics 44~(3)
  (2004) 503--514.

\bibitem{lopez2013generalized}
M.~L{\'o}pez-Fern{\'a}ndez, S.~Sauter, Generalized convolution quadrature with
  variable time stepping, IMA Journal of Numerical Analysis 33~(4) (2013)
  1156--1175.

\bibitem{march2015kernel}
W.~B. {March}, B.~{Xiao}, S.~{Tharakan}, C.~D. {Yu}, G.~{Biros}, A
  kernel-independent {FMM} in general dimensions, in: SC '15: Proceedings of
  the International Conference for High Performance Computing, Networking,
  Storage and Analysis, 2015, pp. 1--12.

\bibitem{liao2003beyond}
S.~Liao, {Beyond Perturbation: Introduction to The Homotopy Analysis Method},
  CRC press, 2003.

\bibitem{trujillo1999riemann}
J.~Trujillo, M.~Rivero, B.~Bonilla, {On a Riemann--Liouville generalized
  Taylor's formula}, Journal of Mathematical Analysis and Applications 231~(1)
  (1999) 255--265.

\bibitem{liu2015local}
Z.~Liu, T.~Wang, G.~Gao, {A local fractional Taylor expansion and its
  computation for insufficiently smooth functions}, East Asian Journal on
  Applied Mathematics 5~(2) (2015) 176--191.

\bibitem{osler1971taylor}
T.~J. Osler, Taylor's series generalized for fractional derivatives and
  applications, SIAM Journal on Mathematical Analysis 2~(1) (1971) 37--48.

\bibitem{tongke2019fractional}
W.~Tongke, F.~Meng, {Fractional order degenerate kernel methods for Fredholm
  integral equations of the second kind with endpoint singularities},
  Mathematica Numerica Sinica 41~(1) (2019) 66.

\bibitem{guo2020fractional}
J.~Guo, T.~Wang, Fractional {Hermite} degenerate kernel method for linear
  {Fredholm} integral equations involving endpoint weak singularities, Journal
  of Applied Analysis \& Computation 10~(5) (2020) 1918--1936.

\bibitem{zarei2018solving}
E.~Zarei, S.~Noeiaghdam, {Solving generalized Abel's integral equations of the
  first and second kinds via Taylor-collocation method}, arXiv preprint
  arXiv:1804.08571 (2018).

\bibitem{toutounian2014new}
F.~Toutounian, H.~Nasabzadeh, {A new method based on generalized Taylor
  expansion for computing a series solution of the linear systems}, Applied
  Mathematics and Computation 248 (2014) 602--609.

\bibitem{Wanner1996}
G.~Wanner, E.~Hairer, {Solving Ordinary Differential Equations}, Springer
  Berlin Heidelberg, 1996.

\bibitem{Srivastava2013}
H.~M. Srivastava, R.~G. Buschman, {Theory and Applications of Convolution
  Integral equations}, Springer Science {\&} Business Media, 2013.

\bibitem{Corduneanu1973}
C.~Corduneanu, {Integral Equations and Stability of Feedback Systems}, Academic
  Press, 1973.

\bibitem{Tamarkin1930}
J.~D. Tamarkin, On integrable solutions of {A}bel's integral equation, Annals
  of Mathematics 31~(2) (1930) 219--229.

\bibitem{jaswon1977integral}
M.~A. Jaswon, Integral equation methods in potential theory and elastostatics
  (1977).

\bibitem{jiang2004second}
S.~Jiang, V.~Rokhlin, Second kind integral equations for the classical
  potential theory on open surfaces {II}, Journal of Computational Physics
  195~(1) (2004) 1--16.

\bibitem{levinson1960nonlinear}
N.~Levinson, A nonlinear {V}olterra equation arising in the theory of
  superfluidity, Journal of Mathematical Analysis and Applications 1~(1) (1960)
  1--11.

\bibitem{Heiden2013}
U.~an~der Heiden, Analysis of Neural Networks, Vol.~35, Springer Science \&
  Business Media, 2013.

\end{thebibliography}

\end{document}